\documentclass{amsart}
\usepackage{amsmath, amsthm,upref, amssymb,pdfsync}
\usepackage{blindtext}
\usepackage{enumitem}

\input xypic
\xyoption{all}

\theoremstyle{plain} 
\newtheorem{theorem}[subsection]{Theorem}
\newtheorem{proposition}[subsection]{Proposition}
\newtheorem{lemma}[subsection]{Lemma}
\newtheorem{corollary}[subsection]{Corollary}

\theoremstyle{definition}

\newtheorem{definition}[subsection]{Definition}

\theoremstyle{remark}

\newtheorem*{remark*}{Remark}

\numberwithin{equation}{subsection}
\newcommand{\set}[1]{\{#1\}}

\newcommand{\mset}[1]{\set{\,#1\,}}

\newcommand{\pair}[1]{\langle #1\rangle}

\newcommand{\mpair}[1]{\pair{\,#1\,}}

\DeclareMathOperator{\sgn}{{\mathrm{sgn}}}
\DeclareMathOperator{\cone}{{\mathrm{cone}}}

\DeclareMathOperator{\Hom}{{\mathrm{Hom}}}

\usepackage{color}
\definecolor{Mcolor}{rgb}{0,0,1}
\definecolor{Wcolor}{rgb}{1,0,0}
\definecolor{lightgray}{rgb}{0.6,0.6,0.6}

\begin{document}

\title[Oriented matroid structures from root systems]{Oriented matroid structures\\ from  realized root systems}

\author{Matthew Dyer}
\address{Department of Mathematics\\ 255 Hurley Building \\ University of Notre Dame \\Notre Dame \\ Indiana 46556 \\U.S.A.}
\email{dyer@nd.edu}
\author{Weijia Wang}
\address{School of Mathematics (Zhuhai)
\\ Sun Yat-sen University \\
Zhuhai, Guangdong, 519082 \\ China}
\email{wangweij5@mail.sysu.edu.cn}

\begin{abstract}
This paper investigates the question of uniqueness of the reduced  oriented matroid structure arising  from   root systems of a Coxeter group in real vector spaces. We settle the question for finite Coxeter groups,  irreducible affine Weyl groups and all rank three Coxeter groups. In these cases, the oriented  matroid structure  is unique unless $W$ is of type $\widetilde{A}_n, n\geq 3$, in which case there are three possibilities.
\end{abstract}

\maketitle

\section{Introduction}
Given a Coxeter system $(W,S)$, one can associate to it different  root
systems $\Phi$ in a real vector space, parameterized by (possibly non-integral)
generalized Cartan matrices (NGCMs for short). Such root systems  play crucial roles in
understanding various mathematical structures, particularly those arising from Lie theory. They naturally determine oriented matroids in the sense of \cite{largeconvex}. One may ask if  different (possibly non-reduced) root system realizations  yield non-isomorphic reduced  oriented matroid structures when transferred to the abstract root system $T\times \{\pm 1\}$ with natural $W$-action, where $T$ is the set of reflections in $W$ (see \cite{Bourbaki}, \cite{bjornerbrenti}).
The transfer is defined as follows. Let $\Phi$ be a realized root system with the oriented matroidal closure operator $\cone_{\Phi}(\Gamma)=\cone(\Gamma)\cap \Phi, \Gamma\subseteq  \Phi.$ Here $\cone(\Gamma)=\{\sum_{i\in I}k_iv_i|v_i\in \Gamma\cup\{0\}, k_i\in \mathbb{R}_{\geq 0}, |I|<\infty\}$. We have a canonical $W\times\set{\pm 1}$-equivariant surjection $\theta\colon\Phi\to T\times \{\pm 1\}$  given by
$$\epsilon\alpha\mapsto (s_{\alpha},\epsilon),\qquad \alpha\in \Phi^+,\epsilon\in \{\pm 1\}.$$
Then we transfer $\cone_{\Phi}$  to an  oriented matroid  closure operator $c_{\Phi}$ on $T\times \{\pm 1\}$ given by $c_{\Phi}(A)=\theta(\cone_{\Phi}(\theta^{-1}A))$ for $A\subseteq T\times\set{\pm 1}$ and consider  the corresponding (reduced) oriented matroid structure on $T\times\set{\pm 1}$.

In this paper we show that for a finite Coxeter group, an
 irreducible  affine Weyl group or a rank 3 Coxeter group, this induced oriented matroid structure is independent of the realized root system $\Phi$ except for type $\widetilde{A}_n, n\geq 3$. For an affine Weyl group of type $\widetilde{A}_n, n\geq 3$, there are 3 different structures. Type $\widetilde{A}_n$ is the most interesting case and is the main focus of this paper. We show that after rescaling, realized root systems of type $\widetilde{A}_n$ can be parameterized by a positive real number $v$.
We then give an explicit general form of the values of chirotope maps as a (Laurent polynomial) function of $v$. For rank three Coxeter systems, we prove the uniqueness by an argument involving homotopies of root systems. Finally, we give an example, touching on our motivations for this work,
illustrating the fact that  reflection orders
(\cite{bjornerbrenti}) don't relate well to the geometry of realized root systems, and ask how they relate to possibly non-realizable oriented matroid structures on root systems.

\section{Coxeter groups, realized root systems and oriented matroids}

In this section, we collect basic concepts of Coxeter groups and oriented matroids.

\subsection{Oriented matroids} There are many equivalent axioms of oriented matroids. In this paper we emphasize the one describing an oriented matroid as an involuted set with a closure operator as given in \cite{om} and \cite{largeconvex}. An oriented matroid is a triple $(E,*,cx)$  where $E$ is a set with an involution map $*: E\rightarrow E$ (i.e. $x^{**}=x, x\neq x^*$) and $cx$ a closure operator on $E$ such that
(1) if $x\in cx(X)$ there exists a finite set $Y\subseteq  X$ such that $x\in cx(Y)$,
(2) $cx(X)^*=cx(X^*)$,
(3) if $x\in cx(X\cup \{x^*\})$ then $x\in cx(X)$,
(4) if $x\in cx(X\cup \{y^*\})$ and $x\not\in cx(X)$ then $y\in cx(X\backslash\{y\}\cup\{x^*\})$.

For each oriented matroid $(E,*,cx)$, define $\overline{E}=E/\sim$ where the equivalence relation is that $x\sim y$ if and only if  $x^*=y$ or $x=y$. Let $\overline{cx}(X)=cx(X\cup X^*)/\sim$. Then $(\overline{E},\overline{cx})$ is an unoriented matroid. Then rank of the oriented matroid of $(E,*,cx)$ is defined to be the rank of the unoriented matroid $(\overline{E},\overline{cx})$.

Let $V$ be a vector space. Let $\Gamma$ be a nonempty set such that $\Gamma=-\Gamma$ and $0\not\in \Gamma$. Then $(\Gamma,-,\cone_{\Gamma})$ is an oriented matroid (where $\cone_{\Gamma}$ is defined in the previous section) with rank being $\dim(\mathbb{R}(\Gamma))$. Such an oriented matroid is said to be realizable. It is known that not every oriented matroid is realizable.

It is known (cf.  \cite[Exercise 3.13]{ombook}) that a finite rank  oriented matroid is determined up to isomorphism by the so-called chirotope map $\chi: E^r\rightarrow \{+,-,0\}$, up to a sign. If $(E,*,cx)$ is realizable of finite rank $r$ then the chirotope map is given by $\chi(e_1,e_2,\ldots,e_r)=\sgn(\det(e_1,e_2,\ldots,e_r)),$ i.e. the sign of the determinant of the $r\times r$ matrix formed by these $r$ vectors (regarded as column vectors with respect to a chosen basis of the ambient space).
The oriented matroid determines two chirotope maps,  depending on the orientation of the chosen basis and differing only by a sign.

\subsection{Coxeter systems} Let $W$ be a group and $S$ be a set of generators of $W$ such that every element of $S$ has order 2 in $W$. For $s,s'\in S,$ denote $m(ss')$ the order of $ss'$. We call $(W,S)$ a Coxeter system if $S$ and the relations $(ss')^{m(ss')}$ with $s,s'\in S, m(ss')\neq \infty$ form a presentation of $W$. $W$ is called a Coxeter group. The elements of $S$ are called simple reflections. Call $|S|$ the rank of the Coxeter system $(W,S)$. In this paper, we assume unless otherwise stated that   $|S|$ is finite, but such Coxeter systems may contain infinite rank reflection subgroups.

Given a Coxeter system $(W,S)$ we associate to it a Coxeter graph in the following way: the set of vertices is $S$ and $s,r\in S$ have an edge between them if $m(rs)\geq 3$ (including $m(rs)=\infty$). If $m(rs)\geq 4$ we label the edge with that number. If the Coxeter graph is connected, $(W,S)$ (and the Coxeter group $W$) is called irreducible and it is called reducible otherwise.
An element of $W$ is called a reflection if it is conjugate to some $s\in S.$ The set of reflections is denoted $T$. Let $w\in W.$ If $w=s_1s_2\cdots s_k, s_i\in S$ with $k$ minimal, we say the length of $w$, denoted $l(w)$, is $k$.

There are complete classifications by Coxeter graphs of finite irreducible Coxeter systems and of a family of infinite Coxeter groups,  called  irreducible  affine Weyl groups, arising from affine reflections in Euclidean spaces. For details, see  \cite{Hum}, \cite{Bourbaki}  and \cite{bjornerbrenti}.

Recall that an $R$-indexed  Coxeter matrix is an $R\times R$-indexed family $(m_{r,s})_{r,s\in R}$  such that $m_{r,r}=1$ for $r\in R$ and $m_{r,s}=m_{s,r}\in \mathbb{N}_{\geq 2}\cup\set{\infty}$ for $r\neq s$ in $R$.
Associated to  a Coxeter system $(W,S)$, there is an $S$-indexed Coxeter matrix $(m_{r,s})_{r,s\in S}$ determined by $m_{r,s}:=m(rs)\in \mathbb{N}_{\geq 1}\cup\set{\infty}$.  An isomorphism $(W,S)\to (W',S')$ of Coxeter systems is a group isomorphism $W\to W'$ which restricts to a bijection $S\to S'$. If $(W,S)$, $(W',S')$ are Coxeter systems, a map $f\colon S\to S'$ extends to an isomorphism $(W,S)\to (W',S')$ if and only if it is a bijection which preserves the Coxeter matrices, in the sense that $m(rs)=m(f(r)f(s))$ for all $r,s\in S$.

Let $T'\subseteq T$, the set of reflections. Then the subgroup $W'$ of $W$ generated by $T'$ is also a Coxeter group with the Coxeter system
$(W',S')$ where $S'=\chi(W'):=\{t\in T\cap W'|l(t't)>l(t), \forall t'\in T\cap W'\backslash \{t\}\})$. It can be shown that if $|T'|=2$, $W'$ is a dihedral group.
Such a group is called a dihedral reflection subgroup.  Partially order the set of all dihedral reflection subgroups of $W$ under inclusion. The maximal elements of this poset are called maximal dihedral reflection subgroups.

\subsection{Realized root systems}  Following  \cite{embed}, \cite{Fu} and \cite{rigidity}, we define   a realized root system  datum to be  data $D=(V,V^{\vee},(-,-),\Pi,\Pi^{\vee},\iota)$ such that:

(1) $V,V^{\vee}$ are $\mathbb{R}-$vector space with positively independent subsets $\Pi,\Pi^{\vee}$ respectively (where a subset $\Gamma$ of a real vector space is said to be positively independent if there don't exist some $n\in \mathbb{Z}^{+}$, pairwise distinct $\gamma_{1},\ldots, \gamma_{n}$ in $\Gamma$ and strictly positive scalars $c_{1},\ldots, c_{n}$ such that $\sum_{i=1}^{n}c_{i}\gamma_{i}=0$),

(2) $\iota: \Pi\rightarrow \Pi^{\vee}$ is a bijection; denote $\iota(\alpha)$ by $\alpha^{\vee}$,

(3) $(-,-): V\times V^{\vee}\rightarrow \mathbb{R}$ is a bilinear pairing such that $(\alpha,\alpha^{\vee})=2$ for all $\alpha\in \Pi,$

(4) $(\alpha,\beta^{\vee})\leq 0$, for all $\alpha\neq \beta$ in $\Pi$,

(5) $(\alpha,\beta^{\vee})=0$ if and only if $(\beta,\alpha^{\vee})=0$, for all $\alpha\neq \beta$ in $\Pi$,

(6) $(\alpha,\beta^{\vee})(\beta,\alpha^{\vee})\in \{4\cos^2\frac{\pi}{m}|m\in \mathbb{N},m\geq 2\}\cup [4,\infty)$, for all $\alpha\neq \beta$ in $\Pi$.

We recall the main properties (see op. cit.). For $\alpha\in \Pi$ define $s_{\alpha}:V\rightarrow V$ by $v\mapsto v-(v,\alpha^{\vee})\alpha$. Let $\widehat{W}$ be the subgroup of $GL_{\mathbb{R}}(V)$ generated by $\widehat{S}:=\{s_{\alpha}|\alpha\in \Pi\}$.

(7) The map $\alpha\mapsto s_{\alpha}\colon \Pi\to \widehat{S}$ is injective. Further,  $(\widehat{W},\widehat{S})$ is a Coxeter system
with Coxeter matrix $(m'_{r,s})_{r,s\in \widehat{S}}$ given by $m'_{s_{\alpha},s_{\alpha}}=2$ for $\alpha\in \Pi$, $m'_{s_{\alpha},s_{\beta}}=m$ if
$(\alpha,\beta^{\vee})(\beta,\alpha^{\vee})= 4\cos^2\frac{\pi}{m}$ where $m\in \mathbb{N}_{\geq 2}$, and $m'_{s_{\alpha},s_{\beta}}=\infty$ if $(\alpha,\beta^{\vee})(\beta,\alpha^{\vee})\geq 4$.

By a realized root system datum of a Coxeter system $(W,S)$,
we mean a realized root system datum $D$ in the above sense
together with a specified isomorphism of Coxeter systems
 $(W,S)\to (\widehat W, \widehat S)$, which we generally
 write (and regard)  as an identification even in situations where several
 realized root system datums of  $(W,S)$ are under
 simultaneous consideration.  Fix a realized root system datum of $(W,S)$.

$\Pi$ is called the set of simple roots. $\Phi:=\widehat{W}\Pi$ is called the set of roots. $\Phi^+:=\Phi\cap \cone(\Pi)$ is called the set of positive roots. $\Phi^-:=-\Phi^+$ is called the set of negative roots.
It can be shown that $\Phi=\Phi^+\dot\cup \Phi^-$ (disjoint union). Similarly, one defines $\Phi^{\vee}$, $(\Phi^{\vee})^{\pm}$, $\widehat {S^{\vee}}$, $\widehat W^{\vee}$ etc.
We call $\Phi^{\vee}$ the set of coroots.
There is a natural identifications of $W=\widehat W$ with  $\widehat{W}^{\vee}$. Also, $\iota$ extends to a $W$-equivariant bijection between $\Phi$ and $\Phi^{\vee}.$ When restricted to $\Phi^+,$ $\iota$ gives a bijection between $\Phi^+$ and $(\Phi^{\vee})^+$. Also denote $\iota(\alpha)$ by $\alpha^{\vee}$ for $\alpha\in \Phi.$ Then for $\alpha\in \Phi$ we can define $s_{\alpha}: V\rightarrow V$ whose action follows the same formula as when $\alpha\in \Pi.$ One can prove that $ws_{\alpha}w^{-1}=s_{w(\alpha)}$ for all $w\in {W}.$
 For $w\in W, \det(w)=(-1)^{l(w)}.$ We call $((\alpha,\beta^{\vee}))_{\alpha,\beta\in\Pi}$ a Non-integral Generalized Cartan Matrix (abbreviated NGCM).

 In general, a matrix $A=(c_{\alpha,\beta})_{\alpha,\beta\in \Pi}$ is a NGCM of a Coxeter system with Coxeter matrix
$(m_{\alpha,\beta})_{\alpha,\beta\in \Pi}$ if and only if the following conditions hold for all $\alpha,\beta\in \Pi$. First, $c_{\alpha,\alpha}=2$. Second,  if $\alpha\neq \beta$, then  $c_{\alpha,\beta}\leq 0$ with equality if and only if $m_{\alpha,\beta}=2$.
Third, if  $m_{\alpha,\beta}\neq \infty,1$, then  $c_{\alpha,\beta}c_{\beta,\alpha}=4\cos^{2}\frac{\pi}{m_{\alpha,\beta}}$. Finally, if
 $m_{\alpha,\beta}=\infty$, then  $c_{\alpha,\beta}c_{\beta,\alpha}\geq 4$.

For simplicity, we  refer to $\Phi$ above as a realized root system of $(W,S)$.
In general the realized  root system thus constructed is not reduced, i.e. there might be roots $\beta=c\alpha$ where $c\neq 1, c\in \mathbb{R}_{\geq 0}.$ However,
recall from \cite{bjornerbrenti} and \cite{Bourbaki}  that there is a
$W$-action on $T\times \{\pm 1\}$ given by
$w(t,\epsilon)=(wtw^{-1},\eta(w,t)\epsilon)$ where
$\eta(w,t)=-1$ if $l(tw^{-1})<l(w^{-1})$ and $\eta(w,t)=1$ otherwise. This $W$-action commutes with the surjection $\pi:\Phi\rightarrow T\times \{\pm 1\}: \epsilon\alpha\mapsto (s_{\alpha},\epsilon), \alpha\in \Phi^+, \epsilon\in \{\pm 1\}$. We may use $\pi$ to transfer the oriented matroid $M=(\Phi,-,\cone_{\Phi})$ to a (reduced) oriented matroid
$M':=(T\times \set{\pm 1},-, c=c_{\Phi})$ where for $A\subseteq T\times \set{\pm 1}$,  $c(A):=\pi(\cone_{\Phi}(\pi^{-1}A)))$.

We shall assume for convenience unless otherwise stated
that $\Pi$ spans $V$ and $\Pi^{\vee}$ spans $V^{\vee}$, and that $S$ is finite.  The last assumption makes it possible to describe $M'$ by its chirotope map.

If in the above construction of realized root system for $(W,S)$, we have  $(\alpha,\beta^{\vee})=-2\cos\frac{\pi}{m(s_{\alpha}s_{\beta})}$ for all $\alpha,\beta\in \Pi$ (so $(\alpha,\beta^{\vee})=-2$ if $m(s_{\alpha}s_{\beta})=\infty$) and $\Pi$, $\Pi^{\vee}$ are linearly independent, the realized root system is said to be a standard one.

We remark that the above notion of a realized
 root system is essentially the most general one in
 which one has the standard convex geometric properties of the root system, such as  the decomposition of the root
 system into positive and negative roots spanning pointed
  cones
 which intersect only at the origin. Most classes of
 reflection representation and realized root system
 considered in the literature  can be interpreted in this
 framework.

\subsection{Rescaling} Let $D=(V,V^{\vee},(-,-),\Pi,\Pi^{\vee},\iota)$ be a realized root system datum associated to a Coxeter system $(W,S)$. Define  $D':=(V,V^{\vee},(-,-),\Delta,\Delta^{\vee},\iota')$ where $\Delta=\{d_{\alpha}\alpha|\alpha\in \Pi\}$, $\Delta^{\vee}=\{d_{\alpha}^{-1}\alpha|\alpha\in \Pi\}$, $d_{\alpha}\in \mathbb{R}_{>0}$ for $\alpha\in \Pi$ and $\iota'(d_{\alpha}\alpha)=d_{\alpha}^{-1}\alpha^{\vee}.$

\begin{lemma}\label{rescale}
$D'$ above is also a realized root system datum of the Coxeter system $(W,S)$. Further,  $D$ and $D'$ determine the same oriented matroid structure on $T$.
\end{lemma}

\begin{proof}
It is straightforward to check that $D'$ is also a realized
root system datum. In particular one notes that
$s_{\alpha}=s_{d_{\alpha}\alpha}$ and $w(d_{\alpha}
\alpha)=d_{\alpha}w(\alpha)$.
The  associated Coxeter system
$(\widehat W,\widehat S)$ of $D'$ is  equal to  that of
$D$,   and  is therefore canonically identified with $(W,S)$.  The above also shows the set of rays spanned by the positive roots for $D$ and $D'$ are exactly the same.  Hence the resulting  reduced oriented matroid on $T\times \set{\pm 1}$  is clearly the same for $D$ and $D'$,  by the chirotope characterization of oriented matroids.
\end{proof}

From now on, we shall say that  $D'$ in the above lemma is a realized root system datum obtained by rescaling $D$. We observe that if  $(W,S)$ is of rank at most two, then any  realized root system datum has linearly independent $\Pi$ and $\Pi^{\vee}$, and can be rescaled to one with a symmetric NGCM. It is well known in this case from explicit formulae for the roots (see for example \cite{DyerReflSubgrp} or \cite{Fu}) that the oriented matroid structure on $T\times\set{\pm 1}$ is independent of choice of realized root system datum.

\section{Coxeter groups with a forest as Coxeter graph}\label{s3}

In this section we show that if a Coxeter group has the Coxeter graph of a forest without infinite bonds, then all realized root systems with linearly independent simple roots induce the same oriented matroid structure on $T\times \{1,-1\}$.

\begin{definition}
A root system $(V,V^{\vee},(-,-),\Pi,\Pi^{\vee},\iota)$ is said to have a symmetric NGCM if $(\alpha,\beta^{\vee})=(\beta,\alpha^{\vee})$ for all $\alpha,\beta\in \Pi.$
\end{definition}

\begin{proposition}\label{tree}
 Let $(V,V^{\vee},(-,-),\Pi,\Pi^{\vee},\iota)$  be   a realized root system datum of the Coxeter system $(W,S)$ whose Coxeter graph is a forest. Then
the root system $\Phi$ can be rescaled to a realized root system having a symmetric NGCM.
\end{proposition}

\begin{proof}
It suffices to check the assertion for $(W,S)$ having the Coxeter graph of a tree.
Note that Coxeter graph can be considered as a graph with vertices set being $\Pi.$ Denote $d(\alpha,\beta)$ the distance between $\alpha$ and $\beta$ in the Coxeter graph for $\alpha, \beta\in \Pi.$
We pick and designate an element $\alpha_0\in \Pi$ as the root of the tree. To prove the assertion, we use induction. Let $\Pi_k=\{\alpha\in \Pi|d(\alpha_0,\alpha)\leq k\}\subseteq  \Pi$. Assume that we can rescale $(V,V^{\vee},(-,-),\Pi,\Pi^{\vee},\iota)$ to $(V,V^{\vee},(-,-),\Delta,\Delta^{\vee},\iota')$ where $\Delta=\Delta_k\cup (\Pi\backslash \Pi_k)$, $\Delta^{\vee}=\Delta_k^{\vee}\cup (\Pi^{\vee}\backslash \Pi_k^{\vee})$, $\Delta_k=\{d_{\alpha}\alpha|\alpha\in \Pi_k\}$, $\Delta_k^{\vee}=\{d_{\alpha}^{-1}\alpha^{\vee}|\alpha\in \Pi_k\}$ and $(d_{\alpha}\alpha,d_{\beta}^{-1}\beta^{\vee})=(d_{\beta}\beta,d_{\alpha}^{-1}\alpha^{\vee})$ for all $\alpha,\beta\in \Pi_k.$
Take any $\alpha$ such that $d(\alpha,\alpha_0)=k+1.$ Then one has some unique $\alpha'\in \Pi$ such that $d(\alpha',\alpha_0)=k$ and $d(\alpha',\alpha)=1.$
Let $d_{\alpha'}\alpha'\in \Delta_k$. Denote $d_{\alpha}=d_{\alpha'}\sqrt{\frac{(\alpha',\alpha^{\vee})}{(\alpha,(\alpha')^{\vee})}}$. Then let $\Xi=\Delta_k\cup \{d_{\alpha}\alpha|d(\alpha,\alpha_0)=k+1\}\cup (\Pi\backslash \Pi_{k+1})$ and $\Xi^{\vee}=\Delta_k^{\vee}\cup \{d_{\alpha}^{-1}\alpha^{\vee}|d(\alpha,\alpha_0)=k+1\}\cup (\Pi^{\vee}\backslash \Pi_{k+1}^{\vee})$. Let $\iota''=\iota'$ when acting on $\Xi\backslash \{d_{\alpha}\alpha|d(\alpha,\alpha_0)=k+1\}$ and $\iota''(d_{\alpha}\alpha)=d_{\alpha}^{-1}\alpha^{\vee}$ otherwise. Then $(V,V^{\vee},(-,-),\Xi,\Xi^{\vee},\iota'')$ is a rescaled realized root system datum having the property that $(d_{\alpha}\alpha,d_{\beta}^{-1}\beta^{\vee})=(d_{\beta}\beta,d_{\alpha}^{-1}\alpha^{\vee})$ for all $\alpha,\beta\in \Pi_{k+1}.$ For
sufficiently large $k$, this gives the desired conclusion since  $S$ is assumed finite.
\end{proof}

\begin{corollary}
If the Coxeter graph of $(W,S)$ is a forest and has no infinite bonds (i.e. the order of $ss'$ is finite for any $s,s'\in S$), then all realized root systems of $(W,S)$ with linearly independent simple roots give the same oriented matroid structure on $T\times\set{\pm 1}.$
\end{corollary}
\begin{proof}
By Proposition \ref{tree}, $(V,V^{\vee},(-,-),\Pi,\Pi^{\vee},\iota)$ can be rescaled to a realized root system datum having symmetric Generalized Cartan matrix and the rescaling does not affect the oriented matroid structure by Lemma \ref{rescale}. Combining this with the fact the Coxeter graph has no infinite bonds, one can assume that the realized root system is the standard one.  The assertion follows.
\end{proof}

\subsection{Linearly dependent simple roots}\label{ld} We now discuss the  relationship between realized root systems with  positively independent simple roots and those with linearly independent simple roots (compare \cite{DyerThesis}, \cite[\S6.1]{Krammer}).  Suppose given  a realized root system datum $D:=(V,V^{\vee},(-,-),\Pi,\Pi^{\vee},\iota)$ in which $\Pi$ is linearly independent.
Let  $U\subseteq V$ be a subspace of the left radical of the bilinear form $(-,-)$ such that $U\cap \cone(\Pi)=0$. Let $V':=V/U$ and $\pi\colon V\to V'$ be the canonical surjection.
The restriction of $\pi$ to $\Pi$ is injective and  $\Pi'=\pi(\Pi)$,  is positively independent, by the assumptions on $U$. Define a bijection $\iota'\colon \Pi'\to  \Pi^{\vee}$ by $\iota'(\pi(\alpha))=\iota(\alpha)$. Then $D':=(V',V^{\vee},(-,-),\Pi',\Pi^{\vee},\iota')$  is a realized root system datum with the same NGCM as  $D$. The Coxeter system attached to $D'$ is therefore canonically isomorphic to
$W$. It follows readily that $\pi$ is $W$-equivariant. This implies the root system for $D'$ is $\Phi'=\pi(\Phi)$.  It is easily
seen that any realized root system is isomorphic to one
obtained in this way from a realized root system with linearly
independent simple roots (compare \cite[3.5]{DyerThesis}). Moreover, rescaling $D$ simply has the effect of rescaling $D'$.
We remark that  the the oriented matroid structures on $T\times\set{\pm 1}$ arising by transfer of those on $\Phi$ and $\Phi'$ have ranks $\vert S\vert$ and $\vert S\vert -\dim(U)$, so are not the same  if $U\neq 0$. Analogous facts to those above also apply to the ``dual''root system $\Phi^{\vee}$ in $V^{\vee}$, so for study of  $\Phi$ and $\Phi'$ above, there would be no loss of generality  in assuming from the outset that $\Pi^{\vee}$ is linearly independent.

If under the assumptions of the previous corollary, the standard realized root system datum $D$ is such that there is no non-trivial subspace $U$ of $V$ satisfying the above assumptions, then $D'$ above must be isomorphic to $D$, and it follows that the assumption in the hypotheses of the Corollary that the simple roots be linearly independent can be omitted.
It is well known that for realized root systems of  finite Coxeter groups, the (left) radical of the  form for the standard realized root system datum  $D$ is zero, and for irreducible  affine Weyl groups, the left radical  is one-dimensional, spanned by a vector lying in $\cone(\Pi)$. Thus, no non-trivial subspace $U$ as above exists in either case.
Therefore we have
\begin{corollary}\label{nonA}
For a  finite  Coxeter group  or  an  irreducible  affine Weyl group which is not of type $\widetilde{A}_n$, all realized root systems give  the same oriented matroid structure on $T\times\set{\pm 1}$.
\end{corollary}
For any Coxeter system with two or more components which are affine Weyl groups, the left  radical of the bilinear form for  the standard realized root system datum  $D$    contains a non-trivial subspace which meets the cone spanned by the positive roots only at zero, and so the oriented matroid structures from realized root systems of such a Coxeter system is not unique.

\section{Realized Root Systems of Type $\widetilde{A}_n$}

\subsection{Realized root systems of type $\widetilde{A}_{n}$}
In this section we deal with the case that the Coxeter  system $(W,S)$ is of type $\widetilde{A}_n, (n\geq 2)$. We give explicitly a  formula for values of  the chirotope map and show that for $n\geq 3$ there are three different oriented matroid structures from  realized root systems.
We take  $S=\{s_0,s_1,s_2,\ldots,s_n\}$ and  $$W=\mpair{S|s_0^2=s_1^2=\ldots=s_n^2=(s_0s_1)^3=(s_1s_2)^3=\ldots=(s_ns_0)^3=e}.$$ Then $W$ admits a realized root system as in the previous section.
We assume that $\Pi=\{\alpha_0,\alpha_1,\ldots,\alpha_n\}$ spans $V$ and must have
$$(\alpha_i,\alpha_{i+1}^{\vee})=a_i, (\alpha_{i+1},\alpha_i^{\vee})=\frac{1}{a_i}, 0\leq i\leq n-1$$
$$(\alpha_{n},\alpha_{0}^{\vee})=a_n, (\alpha_0,\alpha_{n}^{\vee})=\frac{1}{a_n}$$
$$(\alpha_i,\alpha_i^{\vee})=2,$$
$$(\alpha_i,\alpha_j^{\vee})=0\,\,\text{otherwise}$$ for certain negative real numbers $a_{0},\ldots, a_{n}$.
By Lemma \ref{rescale} if we rescale the roots by various positive numbers, it does not change the oriented matroid structure of $\Phi.$
Take $v=\sqrt{|a_0a_1a_2\cdots a_{n}|}$. We replace $\alpha_i, 2\leq i\leq n$ with $|a_1a_2\cdots a_{i-1}|\alpha_i$,
$\alpha_i^{\vee}, 2\leq i\leq n$ with $\frac{1}{|a_1a_2\cdots a_{i-1}|}\alpha_i^{\vee}$, $\alpha_0$ with $\frac{|a_1a_2\cdots a_n|}{v}\alpha_0$ and $\alpha_0^{\vee}$ with $\frac{v}{|a_1a_2\cdots a_n|}\alpha_0^{\vee}$.

Thus without loss of generality, we assume that
$$(\alpha_0,\alpha_{1}^{\vee})=-v, (\alpha_1,\alpha_0^{\vee})=-\frac{1}{v}$$
$$(\alpha_i,\alpha_{i+1}^{\vee})=-1, (\alpha_{i+1},\alpha_1^{\vee})=-1, 1\leq i\leq n-1$$
$$(\alpha_{n},\alpha_{0}^{\vee})=-v, (\alpha_0,\alpha_{n}^{\vee})=-\frac{1}{v}$$
$$(\alpha_i,\alpha_i^{\vee})=2,$$
and $(\alpha_i,\alpha_j^{\vee})=0$ otherwise.

One can check that the NGCM $((\alpha_{i},\alpha_{j}^{\vee}))_{i,j=0,\ldots, n}$ has determinant $-(v-v^{-1})^{2}$.
It is well known  that for $v=1$, the left radical of the form $(-,-)$ is one-dimensional, spanned by $\alpha_{0}+\ldots+\alpha_{n}\in \cone(\Pi)$, and the determinant computation shows the (left or right) radical is zero if $v\neq 1$. It follows from the discussion at the end  of Section \ref{s3}
that any realized root system of type $\widetilde A_{n}$ has linearly independent simple roots  and simple coroots, and is isomorphic to one of those considered above for some (in fact, uniquely determined) positive real number  $v$.

We remark that  distinct positive real scalars $v$ afford  non-isomorphic representations of $W$ on the linear span of $\Pi$,  and these may all be viewed as specializations of a generic reflection representation of $W$ on a free $\mathbb{R}[v,v^{-1}]$-module  (or even $\mathbb{Z}[v,v^{-1}]$-module) where $v$ is an indeterminate, with basis $\alpha_{0},\ldots, \alpha_{n}$
(see \cite[Chapter 2]{DyerThesis}). The formulae we give for $v\in \mathbb{R}$ apply equally well in this generic setting, with $v$ regarded instead as an indeterminate.

For $1\leq i\leq j\leq n$ we denote  $$\alpha_{i,j}'=\alpha_0+v(\alpha_1+\alpha_2+\ldots+\alpha_{i-1})+v^{-1}(\alpha_{j+1}+\alpha_{j+2}+\ldots+\alpha_n)$$
$$=s_{\alpha_{j+1}}s_{\alpha_{j+2}}\ldots s_{\alpha_{n}}s_{\alpha_{i-1}}s_{\alpha_{i-2}}\ldots s_{\alpha_1}(\alpha_0)$$
and denote
$$\alpha_{i,j}=\alpha_i+\alpha_{i+1}+\ldots+\alpha_j=s_{\alpha_j}\ldots s_{\alpha_{i+2}}s_{\alpha_{i+1}}(\alpha_i).$$
For all $n\in \mathbb{Z}$, define
  $c_n:=\frac{v^n-v^{-n}}{v-v^{-1}}$ if $v\neq 1 $ and $c_{n}:=n $ if $v=1$. We have $c_{0}=0$, $c_{1}=1$, $c_{-n}=-c_{n}$ and if  $n> 0$,  then $c_n=v^{n-1}+v^{n-3}+\ldots+v^{-n+1}$. (If $v$ is regarded as indeterminate,  then $c_{n}$ is known as  a Gaussian (or quantum) integer).

The following identity is well known and easily verified.
\begin{lemma}\label{identitygauss}
$c_{n+1}c_{m+1}-c_nc_m=c_{n+m+1}$.
\end{lemma}

Next we describe the root system explicitly.

\begin{proposition}\label{rootform}
$\Phi=\{\pm v^m(c_{k+1}\alpha_{i,j}+c_k\alpha_{i,j}')|m,k\in \mathbb{Z},1\leq i\leq j\leq n\}$.
\end{proposition}

\begin{proof}
We first show that if $\alpha$ is a root then $v^m\alpha$, where $m\in \mathbb{Z}$, is a root. It suffices to show that for $\alpha$ simple and  $m=\pm 1.$ For $\alpha_0$ we have $s_{\alpha_1}s_{\alpha_0}(\alpha_1)=v^{-1}\alpha_0$ and $s_{\alpha_n}s_{\alpha_0}(\alpha_n)=v\alpha_0.$ And we have $s_{\alpha_0}s_{\alpha_1}(\alpha_0)=v\alpha_1$,
$s_{\alpha_1}s_{\alpha_2}(v\alpha_1)=v\alpha_2, \ldots, s_{\alpha_{n-1}}s_{\alpha_n}(v\alpha_{n-1})=v\alpha_n$.
$s_{\alpha_0}s_{\alpha_n}(\alpha_0)=v^{-1}\alpha_n$ and $s_{\alpha_n}s_{\alpha_{n-1}}(v^{-1}\alpha_n)=v^{-1}\alpha_{n-1}, \ldots, s_{\alpha_2}s_{\alpha_1}(v^{-1}\alpha_2)=v^{-1}\alpha_1.$

 Since $(s_{\alpha_j}\ldots s_{\alpha_{i+2}}s_{\alpha_{i+1}}(\alpha_i))^{\vee}=s_{\alpha_j}\ldots s_{\alpha_{i+2}}s_{\alpha_{i+1}}(\alpha_i^{\vee})$ we have $\alpha_{i,j}^{\vee}=\alpha_i^{\vee}+\alpha_{i+1}^{\vee}+\ldots+\alpha_j^{\vee}.$ Similarly one sees that
 $(\alpha_{i,j}')^{\vee}=\alpha_0^{\vee}+v^{-1}(\alpha_1^{\vee}+\alpha_2^{\vee}+\ldots+\alpha_{i-1}^{\vee})+v(\alpha_{j+1}^{\vee}+\alpha_{j+2}^{\vee}+\ldots+\alpha_n^{\vee})$.
One checks that $s_{\alpha_{i,j}}(\alpha_{i,j}')=c_1\alpha_{i,j}'+c_2\alpha_{i,j}$ and $s_{\alpha_{i,j}'}(\alpha_{i,j})=c_1\alpha_{i,j}+c_2\alpha_{i,j}'(=-(c_{-1}\alpha_{i,j}+c_{-2}\alpha_{i,j}')).$
Then an inductive argument (together with Lemma \ref{identitygauss}) shows that by acting by  elements from the (infinite) dihedral subgroup $\mpair{s_{\alpha_{i,j}}, s_{\alpha_{i,j}'}}$ on the roots $\alpha_{i,j},\alpha_{i,j}'$ one gets the set
$\{\pm (c_{k+1}\alpha_{i,j}+c_k\alpha_{i,j}')|k\in \mathbb{Z}\}$.

Therefore we see $\Phi\supseteq \{\pm v^m(c_{k+1}\alpha_{i,j}+c_k\alpha_{i,j}')|m,k\in \mathbb{Z},1\leq i\leq j\leq n\}.$
To finish the proof one only needs to check that the right hand side is stable under the action of simple reflections. This may be done  by easy calculations,  noting  the following cases:
$$s_{\alpha_{p}}(c_{k+1}\alpha_{i,j}+c_k\alpha_{i,j}')=
c_{k+1}\alpha_{i,j}+c_k\alpha_{i,j}',\qquad  p\neq i-1,i,j,j+1$$
$$s_{\alpha_{i-1}}(c_{k+1}\alpha_{i,j}+c_k\alpha_{i,j}')=c_{k+1}\alpha_{i-1,j}+c_k\alpha_{i-1,j}',\qquad i\geq 2,$$
$$s_{\alpha_{i}}(c_{k+1}\alpha_{i,j}+c_k\alpha_{i,j}')=c_{k+1}\alpha_{i+1,j}+c_k\alpha_{i+1,j}',\qquad i\neq j,$$
$$s_{\alpha_{i}}(c_{k+1}\alpha_{i,i}+c_k\alpha_{i,i}')=-(c_{-k+1}\alpha_{i,i}+c_{-k}\alpha_{i,i}'), $$
$$s_{\alpha_{j}}(c_{k+1}\alpha_{i,j}+c_k\alpha_{i,j}')=c_{k+1}\alpha_{i,j-1}+c_k\alpha_{i,j-1}',\qquad i\neq j,$$
$$s_{\alpha_{j+1}}(c_{k+1}\alpha_{i,j}+c_k\alpha_{i,j}')=c_{k+1}\alpha_{i,j+1}+c_k\alpha_{i,j+1}', \qquad j\leq n-1,$$
$$s_{\alpha_0}(c_{k+1}\alpha_{1,j}+c_k\alpha_{1,j}')=-v^{-1}(c_{-k}\alpha_{j+1,n}+c_{-k-1}\alpha_{j+1,n}'),\qquad  j\leq n-1,$$
$$s_{\alpha_0}(c_{k+1}\alpha_{1,n}+c_k\alpha_{1,n}')=-(c_{-k-1}\alpha_{1,n}+c_{-k-2}\alpha_{1,n}'),$$
$$s_{\alpha_0}(c_{k+1}\alpha_{j,n}+c_k\alpha_{j,n}')=-v(c_{-k}\alpha_{1,j-1}+c_{-k-1}\alpha_{1,j-1}'), \qquad j\geq 2.$$
\end{proof}

We shall denote $\Phi$ as $\Phi_{(n,v)}$ if necessary to indicate its dependence on $n$ and $v$.
It is easily seen (e.g by symmetry) that for
 a root $\alpha\in \Phi$ and integer $m$,
 $(v^{m}\alpha)^{\vee}=v^{-m}\alpha^{\vee}$, which implies
  $s_{v^{m}\alpha}=s_{\alpha}$. We say that two roots $\alpha,\beta$ in $\Phi$ are
   associate, written $\alpha\sim \beta$, if one of them is a power of $v$ times the
   other. The reflections of $W$ correspond bijectively to the $\sim$-equivalence classes of positive roots.

\begin{lemma}\label{imdr}
The infinite maximal dihedral reflection subgroups of $W$ are the subgroups  $W_{i,j}:=\mpair{s_{\alpha_{i,j}},s_{\alpha_{i,j}'}}$ for $1\leq i\leq j\leq n$. The set of roots of $\Phi$ for which the corresponding reflection lies in $W_{i,j}$ is \begin{equation*}
\Phi_{i,j}:=\{\,\eta v^{m }(c_{k+1}\alpha_{i,j}+c_{k}\alpha'_{i,j})\mid \eta\in \set{\pm 1}, m,k\in \mathbb{Z}\,\}.
\end{equation*}
One has $\Phi=\dot\bigcup_{1\leq i\leq j\leq m}\Phi_{i,j}$.
\end{lemma}

\begin{proof}
We first assume that $v=1$. In this case, the realized root system is the standard reduced root system as defined in \cite{bjornerbrenti} or \cite{Hum}. We denote $\delta=\alpha_0+\alpha_1+\ldots+\alpha_n$. Then $(\delta,\gamma^{\vee})=(\gamma,\delta^{\vee})=0$ for all $\gamma\in \Phi.$ When $v=1$ it is well known (see for example \cite{DyerLehrer}, though it also follows from Proposition \ref{rootform} with $v=1$) that the set of positive roots is \begin{equation*}
\{\alpha_{i,j}+k\delta|1\leq i\leq j\leq n, k\in \mathbb{Z}_{\geq 0}\}\cup \{-\alpha_{i,j}+k\delta|1\leq i\leq j\leq n, k\in \mathbb{Z}_{>0}\}.
\end{equation*} Consider the dihedral reflection subgroup $W_{i,j}=\mpair{s_{\alpha_{i,j}},s_{\delta-\alpha_{i,j}}}$. Since we have $(\alpha_{i,j},(\delta-\alpha_{i,j})^{\vee})=-2$, the group is infinite and $\Phi_{i,j}:=\{\alpha\in \Phi|s_{\alpha}\in W_{i,j}\}=\{\alpha_{i,j}+k\delta|k\in\mathbb{Z}_{\geq 0}\}\cup \{-\alpha_{i,j}+k\delta|k\in\mathbb{Z}_{>0}\}$. One checks that
$\mathbb{R}\Pi_{i,j}\cap \Phi=\Phi_{i,j}$. This guarantees that $W_{i,j}$ is maximal by Remark 3.2 of \cite{Bruhat}.
For $\{i,j\}\neq \{p,q\}$, the dihedral reflection subgroups $\mpair{s_{\epsilon_1\alpha_{i,j}+t\delta},s_{\epsilon_2\alpha_{p,q}+r\delta}}, \epsilon_1, \epsilon_2\in \{1,-1\}, t,r\in \mathbb{Z}$ are finite. Hence the groups $W_{i,j}$ exhaust all maximal dihedral reflection subgroups.

Note $s_{\delta-\alpha_{i,j}}=s_{\alpha_{j+1}}s_{\alpha_{j+2}}\ldots s_{\alpha_{n}}s_{\alpha_{i-1}}s_{\alpha_{i-2}}\ldots s_{\alpha_1}(\alpha_0).$ For arbitrary $v$ the latter is $s_{\alpha_{i,j}'}$. For any $v$, $s_{\alpha_{i,j}}=s_{\alpha_{j}}\ldots s_{\alpha_{i+1}}(\alpha_{i})$. Since the maximal dihedral reflection subgroups are independent of $v$, this proves the first assertion.
The formula for  $\Phi_{i,j}$ follows since its right hand side  is the set of all roots in $\Phi$ in the plane spanned by $\alpha_{i,j}$ and $\alpha'_{i,j}$,
while,  from the preceding proof, $W_{i,j}\{\alpha_{i,j},\alpha'_{i,j}\}$ consists of the roots $\pm (c_{k+1}\alpha_{i,j}+c_{k}\alpha'_{i,j})$.  The final claim is clear  from Proposition \ref{rootform}.
\end{proof}

\subsection{Chirotopes for $\widetilde A_{n}$} Our next goal is to compute the values of  the
chirotope map of the oriented matroid
$(\Phi,-,\cone_{\Phi})$ by calculating the determinant of
$n+1$ roots $\gamma_{0},\ldots, \gamma_{n}\in \Phi$.
More precisely, we calculate the determinant of the linear
operator which maps
$\alpha_{j}\mapsto \gamma_{j}$ for $j=0,\ldots, n$.
Regard $\gamma_j$ as a column vector, with rows indexed by $0,\ldots, n$,  with the  entry in its $i$-th row being the coefficient of $\alpha_{i}$ in $\gamma_{j}$.
Let $(\gamma_{0},\ldots, \gamma_{n})$ be the
$(n+1)\times(n+1)$-matrix, with rows and columns indexed by $0,\ldots, n$, whose $j$-th column  is $\gamma_{j}$ (viewed as column vector); this matrix represents the above linear transformation with respect to the ordered basis $\alpha_{0},\ldots, \alpha_{n}$ of  $V$, and we seek
to determine
$\det(\gamma_0,\gamma_1,\ldots,\gamma_n)$.
 For later use,  observe that  for $1\leq i\leq j\leq n$, $m,k\in \mathbb{Z}$ and $\eta\in \set{\pm 1}$, (the coordinate vector of the) root $\eta v^{m}(c_{k+1}\alpha_{i,j}+c_{k}\alpha'_{i,j})$  is
\begin{equation}
\label{coord}\eta v^{m} (c_{k},vc_{k},\ldots,vc_{k},c_{k+1},\ldots,c_{k+1},v^{-1}c_{k},\ldots,v^{-1}c_{k})^T
\end{equation}
where $T$ denotes transpose, the entry $c_{k}$ is in row $0$, the first $c_{k+1}$ is in row $i$ and the last $c_{k+1}$ is in row $j$.
The main result we shall establish is the following:
\begin{theorem}\label{Adet}
Suppose $\gamma_0, \gamma_1, \ldots, \gamma_{n}$ are $n+1$ roots of $W$ of type $\widetilde{A}_{n}, n\geq 2$. Regard $\gamma_i$ as coordinate vectors as above. Then $\det(\gamma_0,\gamma_1,\ldots,\gamma_n)$ is either $0$ or of the form
$$\mu v^l(v-v^{-1})^{m-1}\Pi_{k=1}^mc_{h_k}$$
where $\mu\in \{1,-1\}$, $m\geq 1$ and $h_{1},\ldots, h_{m}\in \mathbb{Z}_{+}$. \end{theorem}
More precisely, this result holds (with $v$ regarded as an indeterminate) for the generic reflection representation over $\mathbb{R}[v,v^{-1}]$. It would be equivalent to just require $h_{i}\in \mathbb{Z}$ since $c_{n}=-c_{n}$ and $c_{0}=0$.

Observe that if
there is a permutation $\sigma$ of $\{0,\ldots, n\}$ with $\sgn(\sigma)=\epsilon \in \{\pm 1\}$,  an element $w\in W$,  integers $\eta_{i}\in \{\pm 1\}$,   and integers $m_{i}$ such that  with $\gamma'_{i}=\eta_{i}v^{m_{i}}w(\gamma_{\sigma(i)})$ for $i=0,\ldots, n$,
then \begin{equation}\label{reduceqn}
\det(\gamma'_{0},\ldots, \gamma'_{n}))=(-1)^{l(w)}\epsilon \eta_{0}\cdots \eta_{n}v^{m_{0}+\ldots +m_{n}}\det(\gamma_{0},\ldots, \gamma_{n}),
\end{equation}
so  validity of the theorem for $\gamma_{0}\ldots, \gamma_{n}$ is equivalent to its validity for
$\gamma'_{0}\ldots, \gamma'_{n}$.

We begin the proof with a series of results  to deal with the  cases  in which some maximal dihedral root subsystems $\Phi_{i,j}$ contain exactly two of the roots.
\begin{proposition}
Let $A=(a_{ij})$ be a $n\times n$ matrix with $n\geq 2$. Assume that the $i$-th column is of the form
$$(0,\ldots,0,x_i,0,\ldots,0,-1,0\ldots,0)^T$$ for $i=1,\ldots, n$.
Then $\det(A)$ is either 0 or $$\mu\Pi(x_{i})\Pi(x_{j_1}x_{j_2}\ldots x_{j_l}-x_{k_1}x_{k_2}\ldots x_{k_m})$$
where $\mu\in \{1,-1\}$, each $x_q$ appears at most once in this formula and each $l$, $m$ appearing is  positive.
\end{proposition}

\begin{proof}
We prove this by induction. If $n=2$ the matrix can only be
$$\left(
  \begin{array}{cc}
    x_1 & x_2 \\
    -1 & -1 \\
  \end{array}
\right)$$
whose determinant is $x_2-x_1.$

We claim that after suitably permuting the columns and rows this matrix must contain a $p\times p$-block  of the form
$$\left(
  \begin{array}{cccccc}
    a_1 & 0 & 0 & \cdots & 0 & a_p \\
    a_1' & a_2 & 0 &  \cdots & 0 & 0 \\
    0 & a_2' & a_3 &  \cdots & 0 & 0 \\
    \cdots & \cdots & \cdots & \cdots & \cdots & \cdots \\
    0 & 0 & 0 &  \cdots & a_{p-1} & 0 \\
    0 & 0 & 0 &  \cdots & a_{p-1}' & a_p' \\
  \end{array}
\right)$$
where $p\geq 2$ and  $\{a_q,a_q'\}=\{x_{j_q},-1\}$ for $q=1,\ldots, p$.

To see this, consider a loopless graph (with parallel edges allowed) with vertex set $\{1,\ldots, n\}$ and $n$ edges, joining   $k_{i},l_{i}$ for $i=1,\ldots, n$ where $k_{i}\neq l_{i}$ and  the entries $a_{k_{i},i}$, $a_{l_{i},i}$ of $A$ are non-zero.  Such a graph must have a circuit i.e. distinct edges with endpoints of the form $$\{t_1,t_2\},\{t_2,t_3\},\ldots,\{t_{p-1},t_p\},\{t_1,t_p\}$$
where $p\geq 2$.
For the above block, we  suitably permute the rows and columns of the submatrix of $A$   with rows indexed by $t_{1},\ldots, t_{p}$ and columns corresponding to the above $p$ edges.

The above submatrix has determinant $a_{1}\cdots a_{p-1}a_{p}'-(-1)^{p}a_{1}'\cdots a_{p-1}'a_{p}$. Note that if $a_{1},\ldots, a_{p-1}$ (resp.,  $a_{1}',\ldots, a_{p-1}'$) are all $-1$ then so is $a_{p}$ (resp., $a_{p}'$). It follows the determinant is of the form
 $\mu_1(x_{j_1}x_{j_2}\cdots x_{j_l}-x_{k_1}x_{k_2}\cdots x_{k_m})$ where $l+m=p$, $0<l<p$,  $\mu_1\in \{1,-1\}$  and $j_{1},\ldots, j_{l},k_{1},\ldots, k_{m}$ are pairwise distinct.

After suitably permuting the columns and rows,  $A$ is of the following form
$$\left(
  \begin{array}{cc}
    A_1 & A_2 \\
    0 & B \\
  \end{array}
\right)$$
with $A_1$ having the above form and the  column of $B$  are determined by deleting certain entries from the  columns of $A$ with index unequal to $j_1,\ldots,{j_l},{k_1},\ldots,{k_m}$.  Then $\det(A)=\mu_{1}'\det(B)(x_{j_1}x_{j_2}\cdots x_{j_l}-x_{k_1}x_{k_2} \cdots x_{k_m})$ with $ \mu_{1}'\in \{1,-1\}$.

If $B$ has a column consisting of all zeroes, then $\det(B)=0$ and we are done. So we assume this is not the case.
Now we define a sequence of matrices $\{B_i\}_{i=0}^{\infty}$.
Let $B_0=B$. If every column of $B_i$ has two nonzero entries or $B_i$ is a 1-by-1 matrix or $B_{i}=0$, then $B_{i+1}=B_i$.  Otherwise if column $q$ is the first column from the left which has only one nonzero entry and it is the entry $a_{q,v}$, then $B_{i+1}$ is obtained by removing $q$-th row and $v$-th column from $B_i$. Then eventually we always have $B_j=B_{j+1}=\ldots$ and we denote $B_{\infty}=B_j.$ Then it is clear that $\det(B)=\eta\Pi_{i\in K}x_i\det(B_{\infty}), \eta\in\{1,-1\}$.

$B_{\infty}$ must be a matrix of  one of the following forms:

(1) a zero matrix

(2) a one-by-one matrix with entry $-1$ or $x_{l}$ for some $l$

(3) a matrix with the same property as in the lemma.

Then the lemma follows by induction.
\end{proof}

\begin{corollary}\label{coro1}
Let $A$ be a $n\times n$ matrix with $n\geq 2$ and $i$-th  column  of the form
$$(0,\ldots,0,x_i,0,\ldots,0,-x_i^{-1},0\ldots,0)^T$$ for $i=1,\ldots, n$.
Then $\det(A)$ is either 0 or $$\mu\frac{1}{x_1x_2\ldots x_n}\Pi x_p^2\Pi_{k=1}^m(\Pi_{i\in I_k}x_i^2-\Pi_{j\in J_k}x_j^2)$$
where $\mu\in \{1,-1\}$ and the product $\Pi x_p^2\Pi_{k=1}^m(\Pi_{i\in I_k}x_i^2-\Pi_{j\in J_k}x_j^2)$ contains  each $x_q$  at most once.
\end{corollary}
\begin{corollary}\label{coro2}
Let $A$ be a $n\times n$ matrix with $n\geq 2$. Each column is of the form
$$(0,\ldots,0,v^{a_i},0,\ldots,0,-v^{-a_i},0\ldots,0)^T.$$
Then $\det(A)$ is either 0 or $$\mu v^l(v-v^{-1})^m\Pi_{k=1}^mc_{h_k}$$
where $\mu\in \{1,-1\}$ and $m\in \mathbb{N}$.
\end{corollary}

\begin{proof}
Corollary \ref{coro1} shows that the determinant is zero or
$$\mu v^l\Pi_{k=1}^m(v^{2A_k}-v^{2B_k})$$
$$=\mu v^l\Pi_{k=1}^mv^{A_k+B_k}(v^{-B_k+A_k}-v^{-A_k+B_k})$$
$$=\mu v^{l'}\Pi_{k=1}^m\frac{v-v^{-1}}{v-v^{-1}}(v^{-B_k+A_k}-v^{-A_k+B_k})$$
Clearly the assertion holds for any $v>0$ even though the second equality above requires $v\neq 1.$
\end{proof}

\begin{lemma}\label{prop2rmds} {\rm Theorem \ref{Adet}} holds in the case that there exist $i\neq j$ in $\set{0,\ldots, n}$ such that  $\gamma_{i}=\alpha_0$ and $\gamma_{j}=c_{k_{j}}\alpha_0+c_{k_{j}+1}(\alpha_1+\ldots+\alpha_n)$ for some $k_j\in \mathbb{Z}$.  \end{lemma}
\begin{proof}
Using  \eqref{reduceqn}, we may assume $\gamma_0=\alpha_0$ and $\gamma_n=c_{k_{n}}\alpha_0+c_{k_{n}+1}(\alpha_1+\ldots+\alpha_n)$.

We look at the $n\times n$ lower-right  block $A$ of the matrix $(\gamma_0,\gamma_1,\ldots,\gamma_n)$. Equation \eqref{coord} shows that after factoring out a power of $v$ and $\pm 1$,  each column is of the form
$$(vc_{k_i},\ldots,vc_{k_i},c_{k_i+1},\ldots,c_{k_i+1},v^{-1}c_{k_i},\ldots,v^{-1}c_{k_i})^T.$$
The last column  is
$$(c_{k_n+1},\ldots,c_{k_n+1}).$$

In the subsequent computation we assume $v\neq 1$ and $c_t=\frac{v^t-v^{-t}}{v-v^{-1}}$. The results apply for $v=1$  as well because it is a removable discontinuity of $\frac{v^t-v^{-t}}{v-v^{-1}}$.

Multiplying each column by ${v-v^{-1}}$ gives a matrix $B$. with columns  of the form
$$(v(v^{k_i}-v^{-k_i}),\ldots,v(v^{k_i}-v^{-k_i}),v^{k_i+1}-v^{-k_i-1},\ldots,v^{k_i+1}-v^{-k_i-1},$$$$v^{-1}(v^{k_i}-v^{-k_i}),\ldots,v^{-1}(v^{k_i}-v^{-k_i}))^T.$$
In particular the last column is of the form
$$(v^{k_n+1}-v^{-k_n-1},\ldots,v^{k_n+1}-v^{-k_n-1})^T.$$
And we have $\det(\gamma_0,\gamma_1,\ldots,\gamma_n)=\det(A)=(v-v^{-1})^{-n}\det(B)$

Now we compute the determinant of $B$. To do this, we subtract the first row from each of the rows below. There are three cases:

(1) the last column is

$$((v^{k_n+1}-v^{-1-k_n}),0,\ldots,0)^T$$

(2) the first entry of the column is $v^{k_i+1}-v^{-k_i-1}$

Then the column will become
$$(v^{k_i+1}-v^{-k_i-1},0,\ldots,0,v^{k_i}(v^{-1}-v),\ldots,v^{k_i}(v^{-1}-v))^T$$

(3) the first entry of the column is $v(v^{k_i}-v^{-k_i})$

Then this column will become
$$(v(v^{k_i}-v^{-k_i}),0,\ldots,0,v^{-k_i}(v-v^{-1}),\ldots,v^{-k_i}(v-v^{-1}),$$$$(v^{k_i}-v^{-k_i})(v^{-1}-v),\ldots,(v^{k_i}-v^{-k_i})(v^{-1}-v))^T$$

Now we examine the lower left $(n-1)\times (n-1)$ block of $B$ (denoted $C$). $\det(B)=\rho(v^{k_n+1}-v^{-1-k_n})\det(C)$ where $\rho\in \{1,-1\}$. So $\det(A)=\rho(v^{k_n+1}-v^{-1-k_n})(v-v^{-1})^{-n}\det(C)$.

We factor $v-v^{-1}$ from each column and obtain a $(n-1)\times (n-1)$ matrix $D$ whose columns are either

$$(0,\ldots,0,v^{k_i},\ldots,v^{k_i})^T$$

or

$$(0,\ldots,0,v^{-k_i},\ldots,v^{-k_i},-(v^{k_i}-v^{-k_i}),\ldots,-(v^{k_i}-v^{-k_i}))^T.$$

$\det(A)=\rho(v^{k_n+1}-v^{-1-k_n})(v-v^{-1})^{-1}\det(D)=\rho c_{k_n+1}\det(D)$.

Next we consecutively subtract the 1st, 2nd, 3rd ... row from the rows below and transform $D$ into a matrix with columns of the form either
$$(0,\ldots,0,v^{k_i},0,\ldots,0)^T$$
or
$$(0,\ldots,0,v^{-k_i},0,\ldots,0,-v^{k_i},0\ldots,0)^T.$$
Expanding the resulting determinant  down any column containing  a single entry until no such columns remain, we obtain a matrix to which  Corollary \ref{coro2} applies, proving   the desired formula.
\end{proof}

\begin{lemma}\label{partone} {\rm Theorem \ref{Adet}} holds in the case that there are $i,j$ with $1\leq i\leq j\leq n$ and $p\neq q$ in $\{0,\ldots, n\}$ such  $\gamma_{p},\gamma_{q}\in \Phi_{i,j}$.
\end{lemma}

\begin{proof} Using \eqref{reduceqn}, we may without loss of generality replace the  the roots $\gamma_{i}$ by their conjugates by some $w\in W$, so one of the two roots in $\Phi_{i,j}$ is an associate of $\alpha_0$. Then by Lemma \ref{imdr}, $(i,j)=(1,n)$,  both roots lie in $\Phi_{1,n}$ and the other one must be  associate, up to sign, to  a root $c_{k}\alpha_0+c_{k+1}(\alpha_1+\ldots+\alpha_n)$ for some integer $k$. The Lemma now follows from Lemma \ref{prop2rmds}. \end{proof}

 We next state the last lemma needed for the inductive proof of Theorem  \ref{Adet}.  Recall that we identify roots with their coordinate vectors with respect to the (naturally ordered) simple roots.
\begin{lemma}\label{reducstep} Assume $n\geq 3$ and $1\leq i\leq n$. Let $\gamma\in \Phi\setminus \Phi_{i,i}$.
Regard $\gamma$ as a column vector with rows indexed by $0,\ldots, n$. Let $\beta$ be the column vector obtained by deleting the $i$-th row of $\gamma$. Then $\beta$ is the coordinate vector for some root of the
root system $\Phi_{(n-1,v)}$ of type $\widetilde A_{n-1}$ and parameter $v$.
\end{lemma}
\begin{proof}
Deletion of the entry from row $i$  of the coordinate vector for $\gamma\in \Phi$ as in \eqref{coord} gives a coordinate vector for a root in type $\widetilde A_{n-1}$ unless there is a single
entry $c_{k+1}$ and it occurs in row $i$, which by  Lemma \ref{imdr}, occurs only when  $\gamma\in \Phi_{i,i}$.
\end{proof}
We now state and prove the following more precise version of Theorem \ref{Adet}
\begin{theorem}\label{Adetprec}
Suppose $\gamma_0, \gamma_1, \ldots, \gamma_{n}$ are $n+1$ roots of $W$ of type $\widetilde{A}_{n}, n\geq 2$. Regard $\gamma_i$ as coordinate vectors as above. Then $d:=\det(\gamma_0,\gamma_1,\ldots,\gamma_n)$ is either $0$ or of the form
$$\mu v^l(v-v^{-1})^{m-1}\Pi_{k=1}^mc_{h_k}$$
where $\mu\in \{1,-1\}$, $m\geq 1$ and $h_{1},\ldots, h_{m}\in \mathbb{Z}_{+}$. Further,
if $d\neq 0$ and  $n=2$, then $m=1$, while  if $n\geq 3$, there exist roots $\gamma_{0},\ldots, \gamma_{n}$ for which $d $ is of the  form displayed above  with $m>1$. \end{theorem}
\begin{proof}
We prove the theorem by induction. Let $n=2$. If there are two roots whose corresponding reflections are contained in the same maximal dihedral reflection subgroup then Lemma \ref{partone} ensures the determinant has the desired form. But we still have to show that the power $m-1$ of $(v-v^{-1})$ is zero. Using \eqref{reduceqn}, we may assume the first root  is $\alpha_0\in \Phi_{1,2}$ and the second  is $c_k\alpha_0+c_{k+1}(\alpha_1+\alpha_2)\in \Phi_{1,2}$.
If the third is in $\Phi_{1,2}$, then $d=0$. It suffices to show  the determinants of the following matrices are of the required form:
$$A=\left(
  \begin{array}{ccc}
    1 & c_k & c_{k'} \\
    0 & c_{k+1} & vc_{k'} \\
    0 & c_{k+1} & c_{k'+1} \\
  \end{array}
\right),\quad  B=\left(
  \begin{array}{ccc}
    1 & c_k & c_{k'} \\
    0 & c_{k+1} & c_{k'+1} \\
    0 & c_{k+1} &  v^{-1}c_{k'}\\
  \end{array}
\right)$$
For $A$, note that $vc_{k'}-c_{k'+1}=-v^{-k'}$ so
$\det(A)=v^{-k'}c_{k'+1}$. For $B$, note that $v^{-1}c_{k'}-c_{k'+1}=-v^{k'}$ and $\det(B)=-v^{k'}c_{k+1}$.

Now assume, still with $n=2$, that there are no two roots with corresponding reflections in the same maximal dihedral reflection subgroup. By conjugation one can assume one of the roots is $\alpha_0.$
Then by \eqref{reduceqn}, $\det(\gamma_0,\gamma_1,\gamma_2)$ is the product of a power of $v$, $\pm1$ and the determinant of
$$\left(
    \begin{array}{ccc}
      1 & c_k & c_{k'} \\
      0 & c_{k+1} & vc_{k'} \\
      0 & v^{-1}c_k & c_{k'+1} \\
    \end{array}
  \right).
$$ By Lemma \ref{identitygauss}, it is of the form $\mu v^lc_{h_d}$, proving this case.

Next,  we consider the general case ($n\geq 3$). To show $d$ is of the required form,  we may assume by Lemma \ref{partone}  that there are no two roots whose reflections are contained in the same maximal dihedral reflection subgroup. Conjugating  $\gamma_n$ suitably if necessary, we may  assume using  \eqref{reduceqn} that  $\gamma_{n}$  is equal to a simple root $\alpha_{i}$ where $1\leq i\leq n$. By the assumption,
 none of    $\gamma_{0},\ldots, \gamma_{n-1}$  is in $\Phi_{i,i}$. Expanding the  determinant $\det(\gamma_{0},\ldots, \gamma_{n})$ down the
 $n$-th column   shows $d=\pm\det(\gamma'_{0},\ldots, \gamma'_{n-1})$ where $\gamma'_{j}$ is obtained by deleting the $i$-th row of $\gamma_{j}$. By   Lemma \ref{reducstep},  each $\gamma'_{i}$ for $i=0,\ldots, n-1$ is the coordinate vector with respect to the (ordered) simple  roots of  a root of the root system $\Phi'$ of type $A_{n-1}$ with parameter $v$.
Hence $d$ has the desired form by induction.

It remains to give an example to show $m>1$ can occur  in case $n\geq 4$.
  Let $\gamma_{0}:=c_{2}\alpha_{1,1}+c_{1}\alpha_{1,1}'=\alpha_{0}+(v+v^{-1})\alpha_{1}+
    v^{-1}(\alpha_{2}+\ldots+\alpha_{n})$,
     $\gamma_{i}:=\alpha_{i}$ for $i=1,\ldots, n-2$,
      $\gamma_{n-1}:=c_{2}\alpha_{n,n}+\alpha_{n,n}'=
      \alpha_{0}+v\alpha_{1}+\ldots +v\alpha_{n-1}+(v+v^{-1})
      \alpha_{n}$ and $\gamma_{n}:=\alpha_{n}$. A trivial calculation shows that $d=v-v^{-1}$ in this case.
   \end{proof}

\begin{theorem}
If $W$ is a  finite Coxeter group or an irreducible affine Weyl group which is not of type $\widetilde{A}_n, n\geq 3$, then all of its realizable root systems give the same oriented matroid structure. If $W$ is of type $\widetilde{A}_n, n\geq 3$, there are three oriented matroid structures from its realizable root systems.
\end{theorem}
\begin{proof}  For finite types and affine types
except $\widetilde{A}_{n}$ where $n\geq 1$, this follows from Corollary \ref{nonA}.
It has already been remarked that  the result is well known for $W$ of rank at most two, in particular for $\widetilde{A_{1}}$. For $\widetilde{A}_{n}$ where $n\geq 2$, the simple roots of $T\times \set{\pm 1}$ are $\alpha'_{j}:=(s_{j},1)$ for $j=0,\ldots, n$. We consider the
chirotope $\chi_{(n,v)}$ on  $T\times \set{\pm 1}$ obtained by transfer  from $\Phi_{(n,v)}$, with sign normalized by $\chi_{(n,v)}(\alpha'_{0},\ldots, \alpha'_{n})=+$. The values of the chirotope $\chi_{(n,v)}$ in general  are given by
$\chi_{(n,v)}(w_{0}(\alpha'_{i_{0}}),\ldots, w_{n}(\alpha'_{i_{n}}))=\sgn(\det(w_{0}(\alpha_{i_{0}}),\ldots, w_{n}(\alpha_{i_{n}})))$
for $w_{j}\in W$ and $i_{j}\in \set{0,\ldots, n}$ for $j=0,\ldots, n$. Theorem \ref{Adetprec} ensures that
$\chi_{(2,v)}$ is independent of $v$, and that for $n\geq 3$ and  $v,v'\in \mathbb{R}_{>0}$, one has  $\chi_{(n,v)}=\chi_{(n,v')}$ if and only if $\sgn(v-1)=\sgn(v'-1)$.
\end{proof}

\section{Rank 3 Coxeter Systems}
In this section we show that for a given rank 3  Coxeter system all oriented matroid structures from  realized root systems are the same using an argument involving homotopies of root systems. This has already been briefly sketched in a remark in \cite{DyerWeakOrder}. We here elaborate the proof.

\begin{lemma}\label{pathconnect}
Consider a NGCM as a vector in $\mathbb{R}^{n\times n}$. Then the space of all NGCMs in $\mathbb{R}^{n\times n}$ associated  to a fixed   Coxeter system of rank $n$ is path connected.
\end{lemma}
\begin{proof}

Let $A=(c_{\alpha,\beta})_{\alpha,\beta=1}^{n}$ and $B=
(c'_{\alpha,\beta})_{\alpha,\beta=1}^{n}$ be two NGCMs  in $\mathbb{R}^{n\times n}$ corresponding to the same Coxeter system $(W,S)$ of rank $n$.
Define a map $\phi\colon [0,1]\to \mathbb{R}^{n\times n}$ as follows: for $t\in [0,1]$, $\phi(t)=(c^{(t)}_{\alpha,\beta})_{\alpha,\beta=1}^{n}$ where
$c^{(t)}_{\alpha,\beta}=2$  if $\alpha=\beta$ and otherwise
\begin{equation*}
c^{(t)}_{\alpha,\beta}=-\vert c'_{\alpha,\beta}\vert ^{t}
\vert c_{\alpha,\beta}\vert ^{1-t}
\end{equation*}
and we interpret $0^{t}=0$ for all $t\in [0,1]$.
It is easy to see from above that $\phi$ is continuous, $\phi(t)$  is a NGCM
for all $t\in [0,1]$,  and $\phi(0)=A$, $\phi(1)=B$.
 \end{proof}

\begin{theorem}\label{rank3homotopy}
Let $(W,S)$ be a rank $3$ Coxeter system. Then all oriented matroid structures from realized root systems  of $(W,S)$ are equal.
\end{theorem}

\begin{proof} If a realized  root system of a Coxeter system $(W',S')$ spans a subspace of dimension $d\leq 2$, then $\vert S'\vert =d$. Hence the simple roots (and coroots) of any realized root system of $(W,S)$ are linearly independent, and the realized root system is completely determined by its NGCM.
Let $A,B$ be two NGCMs determining two realized root systems of $(W,S)$.  Then by Lemma \ref{pathconnect} we have a continuous map $\theta:[0,1]\rightarrow \mathbb{R}^{3\times 3}$ such that each $\theta(t)$ is a NGCM of $(W,S)$ and $\theta(0)=A$ and $\theta(1)=B.$ Let $\gamma_1,\gamma_2,\gamma_3\in \Pi$, $w_1,w_2,w_3\in W.$ Let $(w_i(\gamma_i))_t, 1\leq i\leq 3$ be the root $w_i(\alpha_i)$ in the realized root system having NGCM $\phi(t).$ Now we show that if $\det((w_1(\gamma_1))_0,(w_2(\gamma_2))_0,(w_3(\gamma_3))_0)=0$ then $\det((w_1(\gamma_1))_t,(w_2(\gamma_2))_t,(w_3(\gamma_3))_t)=0$ for any $0\leq t\leq 1.$ Note that the assumption implies that $(w_i(\gamma_i))_0, 1\leq i\leq 3$ are linearly dependent. This is equivalent to the condition that either we have some $w_i(\alpha_i)_0=kw_j(\alpha_j)_0, i,j\in \{1,2,3\}$ or $(w_i(\gamma_i))_0, 1\leq i\leq 3$ span a plane. But the former is equivalent to that $\pi(w_i(\alpha_i)_0)(=\pi(w_i(\alpha_i)_t))$ is $\pi(w_j(\alpha_j)_0)(=\pi(w_j(\alpha_j)_t))$ or its negative, which is independent of the realized root system.
The latter is equivalent to the reflection in $\pi(w_i(\alpha_i)_0)(=\pi(w_i(\alpha_i)_t)), 1\leq i\leq 3$ being  contained in a maximal dihedral reflection subgroup. This property is again independent of the realized root system.
Hence we established the claim. Consequently as a continuous function in $t$, $\det((w_1(\gamma_1))_t,(w_2(\gamma_2))_t,(w_3(\gamma_3))_t)$ is either always positive, always zero or always negative.
\end{proof}

\section{Further examples and remarks}

\subsection{Oriented matroid Structures from  symmetric NGCMs}
In this subsection, we give an example showing that for a fixed Coxeter system, even for its  realized root systems  with symmetric NGCMs and linearly independent simple roots,
the corresponding  oriented matroid structures on the abstract root system may differ.
Fix a Coxeter system $(W,S)$ with $S=\set{s_{1},s_{2},s_{3},s_{4}}$ and  Coxeter graph
\begin{equation*}
\xymatrix{
{s_{1}}\ar@{-}[r]^{m}\ar@{-}[d]_{n}&{s_{2}}\ar@{-}[d]^{\infty}\\
{s_{3}}\ar@{-}[r]_{\infty}&{s_{4}}
}
\end{equation*}
where $m,n\in \mathbb{N}_{\geq 3}\cup\set{\infty}$.

We define  realized root system datum $D$ as follows:  $D=(V,V^{\vee},(-,-),\Pi,\Pi^{\vee},\iota)$ where $\Pi=\set{\alpha_{i}\mid i=1,\ldots,4}$
and the associated
NGCM $C=(c_{i,j})_{i,j=1}^{4}$, where $c_{i,j}=(\alpha_{i},\alpha_{j}^{\vee})$, is given by
\begin{equation*}
C=\left(\begin{array}{rrrr}2&-a&-b&0\\
-a&2&0&-c\\
-b&0&2&-d\\
0&-c&-d&2\end{array}\right).
\end{equation*}
where $a=2\cos\frac{\pi}{m}$, $b=2\cos\frac{\pi}{n}$
and $c,d\geq 2$.

The oriented matroid structure of the realized root system is determined by $\chi: \Phi^4\rightarrow \{+,-,0\}, (\gamma_1,\gamma_2,\gamma_3,\gamma_4)\mapsto \sgn(\det(\gamma_1,\gamma_2,\gamma_3,\gamma_4))$ where $(\gamma_1,\gamma_2,\gamma_3,\gamma_4)$ is a $4\times 4$ matrix with $\gamma_i$ is regarded as the column vector with respect to the basis $\{\alpha_1,\alpha_2,\alpha_3,\alpha_4\}$.

One trivially has $\chi(\alpha_{1},\alpha_{2},\alpha_{3},\alpha_{4})=+$, so these chirotopes (for varying $c,d$) cannot differ by a sign.
One has $s_{\alpha_{1}}s_{\alpha_{4}}(\alpha_{2})=a\alpha_{1}+\alpha_{2}+c\alpha_{4}$ and $s_{\alpha_{1}}s_{\alpha_{4}}(\alpha_{3})=b\alpha_{1}+\alpha_{3}+d\alpha_{4}$.
Hence
\begin{equation*}
\chi(s_{\alpha _{1}}s_{\alpha _{4}}(\alpha_{2}), \alpha_{2},\alpha_{3}, s_{\alpha_{1}}s_{\alpha_{4}}(\alpha_{3}))
=\sgn\begin{vmatrix}
a&0&0&b\\
1&1&0&0\\
0&0&1&1\\
c&0&0&d\end{vmatrix}=\sgn(ad-bc). \end{equation*}
Since $a,b>0$ are fixed and $c,d\geq 2$ are arbitrary,
any value of the sign in $\set{+,-,0}$ is possible as $c,d$ vary,
so there are at least three possible chirotopes  $\chi$ arising this way which are distinct even up to sign.

\subsection{A 2-closure biclosed set which is not cone-biclosed}
Let $\Phi$ be a realized root system of Coxeter system $(W,S)$ with simple roots $\Pi$.
We introduce the 2-closure operator $c_{2,\Phi}$ on $\Phi$. A subset $\Gamma$ of $\Phi$ is said to be $c_{2,\Phi}$-closed if for any $\alpha,\beta\in \Gamma$, one has $\{a\alpha+b\beta|a,b\in \mathbb{R}_{\geq 0}\}\cap \Phi\subseteq  \Gamma$.
The $c_{2,\Phi}$-closure of a set is the intersection of all $c_{2,\Phi}$-closed sets containing it.

 We can transfer $c_{2,\Phi}$-closure to a closure operator $c_{2}$ on the abstract root system $T\times \{1,-1\}$ in the same way as we transferred matroid closure operators,  by setting
 $c_{2}(A)=\pi(c_{2,\Phi}(\pi^{-1}A))$ for $A\subseteq T\times \set{\pm 1}$.
  It can be shown  (using for instance \cite[Lemma 4.3(a)]{DyerWeakOrder} and \cite[1.35--1.36 and Proposition 3.1(c)]{rigidity})   that this closure operator $c_{2}$ on $T\times \{1,-1\}$ is independent of the realized root system $\Phi$.
It is studied for its role in the  description of the meet and join of weak order of a Coxeter group,  completion of weak order and its relation to reflection orders and their initial sections. See for example \cite{DyerWeakOrder} and \cite{bjornerbrenti}.

Generally, for any
closure operator $c$ on a set $X$ and any subset $Y$ of $X$,  we say
  that $Z\subseteq Y$ is $c$-biclosed (in $Y$) if $Z$ and
   $Y\backslash Z$ are both $c$-closed. Let $\mathcal{B}_{c}(Y)$ denote the set of all $c$-biclosed subsets of $Y$.
    In particular, this defines $\mathcal{B}_{c_{2}}(T\times
    \set{1})$ and $\mathcal{B}_{c_{\Phi}}(T\times \set{1})$
    where $c_{\Phi}$ is the reduced closure operator on
    $T\times \set{\pm 1}$ obtained by transfer of the
    oriented matroid closure operator $\cone_{\Phi}$ of
    some realized root system $\Phi$ of $(W,S)$. Let $\Psi^{+}:=T\times\set{1}$ be the abstract positive system in $\Psi$. We shall
    call elements of $\mathcal{B}_{c_{2}}(\Psi^{+})$
     (resp., $\mathcal{B}_{c}(\Psi^{+})$ for a fixed realized root system
     $\Phi$) biclosed (resp., biconvex) sets (of $\Psi^{+}$).  From the definitions, one easily sees that
$\mathcal{B}_{c_{\Phi}}(\Psi^{+})\subseteq
 \mathcal{B}_{c_{2}}(\Psi^{+})$.

Most of the rest of this section is devoted to a proof of the following fact.
\begin{theorem}\label{biclosconv} There is  a finite rank Coxeter system $(W,S)$ such that  one has
$\bigcup_{\Phi}  \mathcal{B}_{c_{\Phi}}(\Psi^{+})\subsetneq \mathcal{B}_{c_{2}}(\Psi^{+})$ where the union is over all
realized root systems $\Phi$ of $(W,S)$.
\end{theorem}
\subsection{}  For the  proof of the theorem,
we shall require  a description
of the canonical simple roots of the root subsystem consisting of roots which
are orthogonal to a fixed simple root $\alpha$, in  case  the NGCM associated to a realized root system $\Phi$ is symmetric. This description may be  obtained for  general (even infinite rank) $(W,S)$
by comparison of results from \cite{BrCent} and \cite{BH2} as follows.

Define a groupoid (i.e. a small category in which all morphisms are
invertible) $G$ as follows. The set of objects of $G$ is
$\Pi$. For $\alpha,\beta\in \Pi$, a morphism $\alpha\to \beta$ is a triple
$(\beta,w,\alpha)$ where $w\in W$ satisfies $w(\alpha)=\beta$.
Composition of morphisms in $G$ is determined by the formula
$(\gamma,v,\beta)(\beta,w,\alpha)=(\gamma,vw,\alpha)$.

For any distinct $\alpha,\beta\in \Pi$ with
$m_{\alpha,\beta}:=m(s_{\alpha}s_{\beta})<\infty$, let
$\nu(\beta,\alpha):=w_{\alpha,\beta}s_{\alpha}$ where $w_{\alpha,\beta}$ is
the longest element of the finite standard parabolic subgroup
$\mpair{s_{\alpha},s_{\beta}}$. Then one has a morphism
$\nu'(\beta,\alpha)=(\gamma, \nu(\beta,\alpha),\alpha)\in G$ where
$\gamma:=\nu(\beta,\alpha)\alpha\in \set{\alpha,\beta}$. The following
observations, especially (1), are used frequently below:
\begin{enumerate}
\item[(1)] If $m:=m_{\alpha,\beta}$ is even,
$\nu(\beta,\alpha)=(s_{\beta}s_{\alpha})^{m/2-1}s_{\beta}$ is a reflection in a root orthogonal to $\alpha$, and $\gamma=\alpha$
\item[(2)] If $m$ is odd, then $\nu(\beta,\alpha)
=(s_{\alpha}s_{\beta})^{(m-1)/2}$ is of even length
and $\gamma=\beta$.
\item[(3)] In either case, $\nu'(\beta,\alpha)^{-1}=\nu'(\gamma',\gamma)$ where
$\set{\gamma,\gamma'}=\set{\alpha,\beta}$.
\end{enumerate}

It is easily seen (and a special case of well known results of Howlett and Deodhar)
that the groupoid $G$ is generated by the morphisms $\nu'(\beta,\alpha)$ for distinct
$\alpha,\beta\in \Pi$ with $m_{\alpha,\beta}$ finite. In \cite{BH2}, where relations for $G$
in terms of these generators are given, the generators
$\nu'(\beta,\alpha)\colon \alpha\to \beta$ with $m_{\beta,\alpha}$ odd are called \emph{movers}
and the other generators, of the form $\nu'(\beta,\alpha)\colon \alpha\to \alpha$ with $m_{{\alpha,\beta}}$ even,
are called \emph{shakers}, though ``fixers'' would be more apt in this rank one case. Let $M$ be the
subgroupoid of $G$ generated by all the movers.

For $\alpha\in \Pi$, let $G_{\alpha}:=\Hom_{G}(\alpha,\alpha)$. The map $(\alpha,w,\alpha)\mapsto w$
induces an isomorphism of $G_{\alpha}$ with the stabilizer $\mset{w\in W\mid w\alpha=\alpha}$ of
$\alpha$ in $W$. We regard this as an identification below. Let $M_{\alpha}:=\Hom_{M}(\alpha,\alpha)$,
which is a subgroup of $G_{\alpha}$.
Also, let $S''_{\alpha}:=\mset{\nu'(\beta,\alpha)\mid \beta\in \Pi, m_{\alpha,\beta}\,\text{\rm is even}}$ denote
the set of all shakers with $\alpha$ as domain (and codomain). Set $T'_{\alpha}:=\mset{g^{-1}sg\mid \beta\in \Pi,
g\in \Hom_{G}(\alpha,\beta), s\in S''_{\beta}}$, $W'_{\alpha}:=\mpair{T'_{\alpha}}\subseteq G_{\alpha}$ and
$S'_{\alpha}:=\mset{g^{-1}sg\mid\beta\in \Pi, g\in \Hom_{M}(\alpha,\beta), s\in S''_{\beta}}\subseteq  T'_{\alpha}$.

For $\gamma\in \Phi^+$ define $\gamma^{\perp}=\{\rho\in \Phi\mid (\rho,\gamma^{\vee})=(\gamma,\rho^{\vee})=0\}$.
Also let $T_{\alpha}:=\mset{s_{\beta}\mid \beta\in \Phi\cap \alpha^{\perp}}\subseteq  T$,
$W_{\alpha}=\mpair{T_{\alpha}}\subseteq W$ and $S_{\alpha}:=\chi(W_{\alpha})$.  The point of the following result is to show that two natural sets of Coxeter group generators for the reflection subgroup $W_{\alpha}$ generated by all roots orthogonal to a fixed simple root $\alpha$ coincide.

\begin{proposition}\label{rootcent}
\begin{enumerate}
\item[(a)] $G_{\alpha}=M_{\alpha}\ltimes W'_{\alpha}$ (a semidirect product of groups
with $W'_{\alpha}$ normal in $G_{\alpha}$).

\item[(b)] $M_{\alpha}$ is a free group (isomorphic to the fundamental group based at
$\alpha$ of the ``odd Coxeter graph'', on vertex set $\Pi$, of $(W,S)$).

\item[(c)] $(W_{\alpha}',S_{\alpha}')$ is a Coxeter system with reflections $T'_{\alpha}$
and the natural (conjugation) action of $M_{\alpha}$ on $W'_{\alpha}$ from (a) fixes the set $S'_{\alpha}$ and hence $T'_{\alpha}$.

\item[(d)] One has $W'_{\alpha}=W_{\alpha}$, $T'_{\alpha}=T_{\alpha}$ and $S'_{\alpha}=S_{\alpha}$.
\end{enumerate}
\end{proposition}

\begin{proof} With the above identification of $G_{\alpha}$ with the centralizer of $\alpha$
in $W$, it follows easily from (a) above that $S''_{\alpha}\subseteq  S'_{\alpha}\subseteq
T'_{\alpha}\subseteq T_{\alpha}$ and hence $W'_{\alpha}\subseteq W_{\alpha}$.
From \cite{BrCent}, one has $G_{\alpha}=M_{\alpha}\ltimes W_{\alpha}$ and (b),
while from \cite{BH2}, (as the rank one case of results in any rank) one has
$G_{\alpha}=M_{\alpha}\ltimes W'_{\alpha}$ and (c). This implies $W_{\alpha}=W'_{\alpha}$.

In the proof of (c) in \cite{BH2}, a realized root system $\Upsilon$ for $W'_{\alpha}$
in the real vector space $\alpha^{\perp}$ is constructed,
whose roots are certain positive scalar multiples of projections of certain roots
from $\Phi$ on $\alpha^{\perp}$ and whose set of reflections is $T_{\alpha}'$.
Tracing through the construction of the simple roots $\Delta$ for $\Upsilon$ (we omit details),
it follows using (1) above that the simple root corresponding to the simple reflection
$g^{-1}sg\in S'_{\alpha}$, where $\beta\in \Pi$, $g\in \Hom_{M}(\alpha,\beta)$ and
$s\in S''_{\beta}$ is $g^{-1}(\delta)\in \Phi^{+}$ where $\delta\in \Phi^{+}$ with $s_{\delta}=s$.

It follows that $\Upsilon$ is (in this special rank one situation, though not in general)
a root subsystem of $\Psi$. Further, $\Delta$, as a basis of simple roots for this subsystem
contained in the positive roots $\Phi^{+}$, corresponds to $\chi(W')$ by \cite{DyerReflSubgrp}.
It follows that $S'_{\alpha}=S_{\alpha}$ and hence $T'_{\alpha}=T_{\alpha}$ since
they are both the sets of reflections of $(W_{\alpha},S_{\alpha})=(W'_{\alpha},S'_{\alpha})$. This completes the proof.
\end{proof}

\subsection{} We remark that in general, the listed elements
$g^{-1}sg$ in the set defining $S'_{\alpha}$ are not
necessarily distinct. However, in our application below,
$(W,S)$ has no finite standard parabolic subgroups of rank
three. This implies that, in the terminology of
\cite{BH2}, the above-mentioned presentation of
the groupoid $G$ has no $(R2)$ relations i.e. only involves
$(R1)$ relations (as in (c)). This assures further that the elements
$g^{-1}sg$ in the set defining $S'_{\alpha}$ are pairwise
distinct and that $(W_{\alpha},S_{\alpha})$ has no nontrivial
braid relations i.e. is a universal Coxeter system.

Theorem \ref{biclosconv}  is equivalent to the statement below.
 \begin{proposition} There is a finite rank Coxeter system $(W,S)$ with a $c_{2}$-biclosed
subset $\Xi$ of $\Psi^{+}$ such that, for every realized root system $\Phi$ of $(W,S)$, with canonical surjection $\pi\colon \Phi\to \Psi$,
$\pi^{-1}(\Xi)$ is not a $\cone_{\Phi}$-biclosed subset of $\Phi^{+}$.
\end{proposition}
\begin{proof}
To prove the theorem, we take for $(W,S)$ a finite rank Coxeter system,
mentioned to the first author by Bob Howlett, in which the
reflection subgroup generated by reflections in roots
orthogonal to some (in fact, any) fixed root $\alpha$ is an
infinite rank universal Coxeter group. We then prove the
assertion of the theorem in the case of the standard
root system $\Phi$. Finally, we deduce the assertion of
the theorem for arbitrary $\Phi$ by a homotopy argument similar to that in the proof of Theorem \ref{rank3homotopy}.

Consider the rank four Coxeter system $(W,S)$
with $S=\mset{s_{i}\mid i=1,2,3,4}$ and Coxeter graph
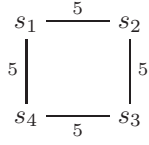
\begin{equation*}
\xymatrix{
{s_{1}}\ar@{-}[r]^{5}\ar@{-}[d]_{5}&{s_{2}}\ar@{-}[d]^{5}\\
{s_{4}}\ar@{-}[r]_{5}&{s_{3}}
}
\end{equation*}
and its standard root system
$\Phi$ with simple roots $\mset{\alpha_{i}\mid 1=1,2,3,4}$.

  To take advantage of the symmetry,
take the indices modulo $4$ i.e set $s_{i+4n}:=s_{i}$ and $\alpha_{i+4n}:=\alpha_{i}$ for any integers $1\leq i\leq 4$ and $n$.

All the roots are $W$-conjugate, so we consider for
definiteness the simple root $\alpha=\alpha_{1}$ and the
corresponding reflection group $W_{\alpha}$ as defined
above. Let $x_{i}:=s_{i}s_{i+1}s_{i}s_{i+1}$ for $i\in \mathbb{Z}$
and $w:= x_{4}x_{3}x_{2}x_{1}$ in $W$. From Proposition
\ref{rootcent}, one sees that $M_{\alpha}=\mpair{w}\cong \mathbb{Z}$ and
that the canonical simple system $\Delta\subseteq  \Phi^{+}$ for $W_{\alpha}$ is
\begin{equation*}
\Delta=\dot\bigcup_{n\in \mathbb{Z}}w^{n}\Delta', \quad \Delta':=\set{\beta_{1}:=\alpha_{3},\beta_{2}:=x_{4}\alpha_{2},
\beta_{3}:=x_{1}^{-1}\alpha_{4},\beta_{4}:=x_{4}x_{3}\alpha_{1}}.
\end{equation*}
In fact, $M_{\alpha}$ acts freely on $\Delta$ with $\Delta'$ as a set of
(four) orbit representatives which we now describe explicitly for subsequent use.

Let $\tau:=2\cos\frac{\pi}{5}=\frac{1+\sqrt{5}}{2}$ be the
golden ratio, so $\tau^{2}=\tau+1$. We have
$s_{\alpha_{i}}\alpha_{i\pm 1}=\alpha_{i\pm 1}+
\tau\alpha_{i}$, $s_{\alpha_{i}}(\alpha_{i+2})=\alpha_{i+2}$
and of course $s_{\alpha_{i}}(\alpha_{i})=-\alpha_{i}$. From
this, one computes $\beta_{1}=\alpha_{3}$,
$\beta_{2}=(1+2\tau)\alpha_{1}+\alpha_{2}+(1+2\tau)
\alpha_{4}$, $\beta_{3}=(1+2\tau)\alpha_{1}+(1+2\tau)
\alpha_{2}+\alpha_{4}$, $\beta_{4}=
(6\tau+4)\alpha_{1}+(2\tau+1) \alpha_{3}+
(8\tau+4)\alpha_{4}$.

Maintain the notation above. We give an example of a $c_{2,\Phi}$-biclosed subset $\Xi'$ of $\Phi^{+}$ which is not
$\cone_{\Phi}$-biclosed.
Note that
$\beta_{3}+\beta_{4}=(2\tau+1)(\beta_{1}+\beta_{2})$.
Also, $s_{\beta_{1}}(\beta_{2})=\beta_{2}+
(4\tau+2)\beta_{1}$ so
$(2\tau+1)s_{\beta_{1}}(\beta_{2})=
\beta_{3}+\beta_{4}+(2\tau+1)(4\tau+1)\beta_{1}$.
This shows that $s_{\beta_{1}}(\beta_{2})\in
\cone(\set{\beta_{1},\beta_{3},\beta_{4}})$.

Now let $\Xi':=\Gamma\cup \Gamma'$
where $\Gamma=\mset{\gamma\in \Phi^{+}\mid (\gamma,\alpha^{\vee})<0}$ and
$\Gamma'$ is any  $c_{2, \Phi_{W_{\alpha}}}$-biclosed subset of  $\Phi_{W_{\alpha}}^{+}$.
 We claim  that $\Xi'\in \mathcal{B}_{c_{2,\Phi}}(\Phi^{+})$. To see this, let $\Gamma'':=\mset{\gamma\in \Phi^{+}\mid (\gamma,\alpha^{\vee})>0}$ and consider $\delta,\beta,\gamma\in \Phi^{+}$ such that   $\gamma=b\beta+d \delta$ for some $b,d>0$. We have to show that if $\delta$ and $\beta$ are in $\Xi'$ (resp., $\Phi^{+}\setminus \Xi'$) then $\gamma$ is in $\Xi'$ (resp.,  $\Phi^{+}\setminus \Xi'$).  But this holds since if either $\delta$ or $\beta$ is in $\Gamma$ (resp., $\Gamma''$), then so is $\gamma$ since $\Phi_{W_{\alpha}}=\mset{\epsilon\in \Phi^{+}\mid (\epsilon,\alpha^{\vee})=0}$, while if both $\delta$ and $\beta$ are in $\Gamma'$ (resp., $\Phi^{+}_{W_{\alpha}}\setminus \Gamma'$) then so is $\gamma$ since $\Gamma'$ is $c_{2,\Phi_{W_{\alpha}}}$-biclosed in $\Phi_{W_{\alpha}}^{+}$.

   Take for $\Gamma'$,
 using \cite[Proposition 2.3]{DyHS1}, the positive root system of the standard parabolic subgroup
$\mpair{s_{\beta_{1}},s_{\beta_{3}},s_{\beta_{4}}}$ of $W_{\alpha}$. Then $\Xi'\in \mathcal{B}_{c_{2},\Phi}(\Phi^{+})$.
However, $\Xi'$ is not $\cone_{\Phi}$-biclosed since
from above, $s_{\beta_{1}}(\beta_{2})\in \cone(\Xi')\cap (\Phi^{+}\backslash \Xi')$.
Note for later use that this argument implies that $\Xi'\cap \Phi_{W_{\alpha}}$
is not $\cone_{\Phi_{W_{\alpha}}}$-biclosed in $\Phi_{W_{\alpha}}^{+}$.

Let $\Psi$ be the abstract root system of
$(W,S)$. We transfer the $\cone_{\Phi}$-biclosed subset $\Xi'$ of $\Phi$
to a subset $\Xi$ of $\Psi$ via the canonical bijection
$\pi\colon \Phi\to \Psi$, setting
$\Xi=\pi(\Xi')=
\mset{(s_{\alpha},1)\mid \alpha\in \Xi'}\subseteq  \Psi$.   Also, let $U:=W_{\alpha}$.

Now we change notation to allow, more
generally, $\Phi$ to be any realized root system of
$(W,S)$ and $\pi\colon \Phi\to \Psi$ to denote the
canonical $W$-equivariant surjection. (Similarly to the case of $\widetilde A_{n},n\geq 2$, such realized root
systems are parametrized up to rescaling and isomorphism by elements of $\mathbb{R}_{>0}$).
Let $\delta_{i}=(s_{i},1)$ for $i=1,\ldots, 4$ be the simple roots of $\Psi$. The oriented
matroid structure on $\Phi$ induces by transfer via
$\pi$ a structure of reduced oriented matroid
on $\Psi$, with closure operator $c_{\Phi}=c$ given by
$c(\Gamma)=\pi(\cone_{\Phi}(\pi^{-1}(\Gamma)))$ for any $\Gamma\subseteq  \Psi$.
As the notation suggests,
$c_{\Phi}$ could possibly depend on the choice of $\Phi$, though we do not know if it actually does (and it does not matter for our purposes).
We shall show however that
\begin{itemize}\item[($*$)] $\Psi_{U}$ is $c_{\Phi}$-closed and the restriction $c_{U}$ of $c_{\Phi}$ to a closure operator on $\Psi_{U}$ is
independent of the choice of $\Phi$.\end{itemize}

This will imply that $\Xi'':=\pi^{{-1}}(\Xi)$ is not
$\cone_{\Phi}$-biclosed in $\Phi^{+}$. For suppose to the
contrary that $\Xi''$ is $\cone_{\Phi}$-biclosed in $\Phi^{+}$. Then
$\Xi''\cap \Phi_{U}$ would be $\cone_{\Phi_{U}}$
biclosed in $\Phi_{U}^{+}$. The corresponding subset
$\pi(\Xi''\cap \Phi_{U})=\Xi\cap \Psi_{U}$ would then be $c_{\Phi}$-biclosed (i.e. $c_{U}$-biclosed) in
$\Psi_{U}$. But then it would be $c_{\Phi}$-biclosed in the case $\Phi$ is the
standard root system of $(W,S)$ (by ($*$)) and we have seen already that it is not.

We now turn to the proof of the key point ($*$). Denote the simple root of $\Phi$
corresponding to $s_{i}$ as $\alpha_{i}$ as before, for $i=1,\ldots, 4$.
We have \begin{equation*}\begin{split}
\Phi_{U}&=\pi^{-1}(\Psi_{U})=\mset{\beta\in \Phi\mid s_{\beta}(\alpha_{1})\in \mathbb{R}_{>0}\alpha_{1}}=
\mset{\beta\in \Phi\mid s_{\beta}(\alpha_{1})=\alpha_{1}}
\\&=\mset{
\beta\in \Phi\mid (\alpha_{1},\beta^{\vee})=0}
=\mset{
\beta\in \Phi\mid (\beta,\alpha_{1}^{\vee})=0}.
\end{split}\end{equation*}

So $\Phi_{U}$ is clearly a flat of $(\Phi,*,\cone_{\Phi})$, which implies $\Psi_{U}=\pi(\Phi_{U})$
is a flat of the corresponding reduced oriented matroid $(\Psi,*,c)$.
In particular, $\Phi_{U}$ is $c$-closed.

Next we observe that, since $U$ is an infinite rank reflection subgroup, it is not of rank two or less and its
root system $\Phi_{U}$ must therefore span a subspace
of dimension three or greater. Since that subspace does
not contain $\alpha_{1}$, it follows that its dimension is
exactly three and that the simple roots of $\Phi$ must
be linearly independent.

Hence $\Phi$ is completely determined, as a subset
of the real vector space with
$\Pi=\mset{\alpha_{i}\mid i=1,2,3,4}$ as basis, by its
NGCM. Suppose $A$ and $B$ are two NGCMs for
$(W,S)$. Choose by Lemma \ref{pathconnect} a continuous map
$\theta\colon [0,1]\to \mathbb{R}^{4\times 4}$ with
NGCMs for values, with $\theta(0)=A$ and
$\theta(1)=B$. We consider four roots $\gamma_{i}=w_{i}(\delta_{j_{i}})$
for $i=1,2,3,4$ such that $w_{1}=1$, $j_{1}=1$ (so
$\gamma_{1}=\delta_{1}$) and
$\gamma_{2},\gamma_{3},\gamma_{4}\in \Psi_{U}$.
We have corresponding roots
$(w_{i}(\alpha_{j_{i}}))_{t}$ in the root system  $\Phi(t)$
with NGCM $\theta(t)$ and simple roots
 $\alpha_i$ for $i=1,\ldots 4$.
We claim that the sign of $a_{t}:=\det( (w_{1}(\alpha_{j_{1}}))_{t},
\ldots,w_{4}(\alpha_{j_{4}}))_{t} )\in \mathbb{R}$ is independent of $t\in [0,1]$.

Observe that  $(w_{i}(\alpha_{j_{i}}))_{t}\in \Phi(t)_{U}$  if
and only if $i=2,3,4$. Since
 $\Phi(t)_{U}=\{\, \beta\in \Phi(t)\mid (\beta, \alpha_{1}^{\vee})=0\}$
 is the intersection of  $\Phi(t)$
with a subspace of dimension three, one has $a_{t}=0$ if and only if two
of the roots $(w_{i}(\alpha_{j_{i}}))_{t}$ for $i=2,3,4$
span the same ray (i.e. two abstract roots
$(w_{i}(\delta_{j_{i}}))$ for $i=1,\ldots, 4$ coincide or differ only by sign) or these
three roots span a plane (i.e. the corresponding
abstract roots $(w_{i}(\delta_{j_{i}}))$ for $i=2,3,4$ lie in a maximal dihedral root
subsystem of $\Psi$). Since the formulations in terms of $\Psi$ are independent of $t$,
it follows that if $a_{t}=0$ for one $t\in [0,1]$, then $a_{t}=0$ for all $t\in [0,1]$.
Since the map $t\mapsto a_{t}$ is continuous, the sign $\sgn(a_{t})$ is independent
of $t$. Since $\gamma_{2},\gamma_{3},\gamma_{4}\in \Psi_{U}$ are arbitrary, it follows that the map
$(w_{2}(\delta_{j_{2}}),
\ldots,w_{4}(\delta_{j_{4}}))\mapsto \sgn(\det((w_{1}(\alpha_{j_{1}}))_{t},
\ldots,w_{4}(\alpha_{j_{4}}))_{t} )$ defines a chirotope for $\Psi_{U}$ corresponding
to the restriction of  $c_{\Phi(t)}$  to a closure operator on $\Psi_{U}$ and that
it is independent of $t$ as claimed. In particular, any two NGCMs $A$, $B$ for $(W,S)$
give the same oriented matroid $(\Psi_{U},*,c_{U})$ as asserted.
\end{proof}
\subsection{Oriented matroid root systems} One may regard Theorem \ref{biclosconv} as   showing that  natural combinatorially defined notions
of $c_{2}$-biclosed sets and initial sections of reflection orders do not correspond well to the geometry of realized root systems.

 Define a (reduced) oriented matroid root system for a Coxeter system $(W,S)$ to be  a reduced oriented matroid
 $(\Psi,-,d)$   on the abstract root system $\Psi$ of $(W,S)$
  such that $d$ is $W$-equivariant (i.e. $W$ acts on
  $\Psi$ by automorphisms of this oriented matroid),
  $\Psi^{+}$ is $d$-closed and  for each maximal dihedral
  reflection subgroup $W'$ of $W$, the corresponding root
   subsystem $\Psi_{W'}:=(W'\cap T)\times \set{\pm 1}$ of $\Psi$ is $d$-closed. It is easily seen that any $(\Psi,-,c_{\Phi})$ is an oriented matroid root system in this sense; the first author will give basic properties of oriented matroid root systems in general elsewhere. It is natural to ask whether one may have, say, $\mathcal{B}_{c_{2}}(\Psi^{+})=\mathcal{B}_{d}(\Psi^{+})$ for some such $d$
   or $\mathcal{B}_{c_{2}}(\Psi^{+})=\bigcup_{d}\mathcal{B}_{d}(\Psi^{+})$ where the union is over all oriented matroid root systems $(\Psi,-,d)$. An even more basic open question is whether there exist any non-realizable $d$ (i.e.  $d$ as above which are not of the form $c_{\Phi}$ for any realizable root system $\Phi$). One may conjecture on the basis of Theorem \ref{rank3homotopy} that there is a unique oriented matroid root system for $(W,S)$ of rank three.
 However, we conjecture that  there are finite rank Coxeter systems for which such non-realizable $d$ exist (possibly even  in great abundance).  The question of whether they exist for irreducible  affine Weyl groups is an interesting and accessible one.


\begin{thebibliography}{10}

\bibitem{bjornerbrenti}
A.~Bj\"{o}rner and F.~Brenti.
\newblock {\em Combinatorics of {C}oxeter Groups, volume 231 of Graduate Texts
  in Mathematics}.
\newblock Springer, 2005.

\bibitem{ombook}
A.~Bj\"{o}rner, Michel~Las Vergnas, Bernd Sturmfels, Neil White, and G.~M.
  Ziegler.
\newblock {\em Oriented Matroids,volume 46 of Encyclopedia of Mathematics and
  its Applications.}
\newblock Cambridge University Press, 1999.

\bibitem{Bourbaki}
N.~Bourbaki.
\newblock {\em \'{E}l\'ements de math\'ematique. {F}asc. {XXXIV}. {G}roupes et
  alg\`ebres de {L}ie. {C}hapitre {IV}: {G}roupes de {C}oxeter et syst\`emes de
  {T}its. {C}hapitre {V}: {G}roupes engendr\'es par des r\'eflexions.
  {C}hapitre {VI}: syst\`emes de racines}.
\newblock Actualit\'es Scientifiques et Industrielles, No. 1337. Hermann,
  Paris, 1968.

\bibitem{BrCent}
Brigitte Brink.
\newblock On centralizers of reflections in {C}oxeter groups.
\newblock {\em Bull. London Math. Soc.}, 28(5):465--470, 1996.

\bibitem{BH2}
Brigitte Brink and Robert~B. Howlett.
\newblock Normalizers of parabolic subgroups in {C}oxeter groups.
\newblock {\em Invent. Math.}, 136(2):323--351, 1999.

\bibitem{largeconvex}
J.~Richard Buchi and William~E. Fenton.
\newblock Large convex sets in oriented matroids.
\newblock {\em Journal of Combinatorial Theory, Series B}, 45:293--304, 1988.

\bibitem{DyerLehrer}
M.~Dyer and G.~Lehrer.
\newblock Reflection subgroups of finite and affine weyl groups.
\newblock {\em Transactions of the Amerian Mathematical Society},
  363(11):5971--6005, 2011.

\bibitem{DyerThesis}
M.~J. Dyer.
\newblock {\em {Hecke} Algebras and Reflections in {Coxeter} Groups}.
\newblock PhD thesis, Univ. of Sydney, 1987.

\bibitem{DyerReflSubgrp}
M.~J. Dyer.
\newblock Reflection subgroups of {C}oxeter system.
\newblock {\em J. Algebra}, 135:57--73, 1990.

\bibitem{DyHS1}
M.~J. Dyer.
\newblock Hecke algebras and shellings of {B}ruhat intervals.
\newblock {\em Compositio Math.}, 89(1):91--115, 1993.

\bibitem{rigidity}
Matthew Dyer.
\newblock On rigidity of abstract root systems of {C}oxeter systems.
\newblock preprint(2010), arXiv:1011.2270.

\bibitem{DyerWeakOrder}
Matthew Dyer.
\newblock On the weak order of {C}oxeter groups.
\newblock preprint(2011), arXiv:abs/1108.5557.

\bibitem{Bruhat}
Matthew Dyer.
\newblock On the ``bruhat graph'' of a coxeter system.
\newblock {\em Compositio Math.}, 78(2):185--191, 1991.

\bibitem{embed}
Matthew~J. Dyer.
\newblock Embeddings of root systems i: Root systems over commutative rings.
\newblock {\em Journal of Algebra}, 321:3226--3248, 2009.

\bibitem{om}
Jon Folkman and Jim Lawrence.
\newblock Oriented matroids.
\newblock {\em Journal of Combinatorial Theory, Series B}, 5:263--288, 1990.

\bibitem{Fu}
Xiang Fu.
\newblock Nonorthogonal geometric realizations of {C}oxeter groups.
\newblock {\em J. Aust. Math. Soc.}, 97(2):180--211, 2014.

\bibitem{Hum}
J.~Humphreys.
\newblock {\em Reflection groups and Coxeter groups, volume 29 of Cambridge
  Studies in Advanced Mathematics,}.
\newblock Cambridge University Press, 1990.

\bibitem{Krammer}
Daan Krammer.
\newblock The conjugacy problem for {C}oxeter groups.
\newblock {\em Groups Geom. Dyn.}, 3(1):71--171, 2009.

\end{thebibliography}

\end{document}